\documentclass[11pt,a4pape,reqno]{amsart}
\usepackage{amsthm,amsmath,amssymb,epsfig,graphicx}
\usepackage{hyperref}

\usepackage{epstopdf}

\usepackage[T1]{fontenc}  
\usepackage[utf8]{inputenc} 
\usepackage{color}
\usepackage{enumerate}
\usepackage{setspace}
\usepackage{amsmath}
\usepackage{amsthm}
\usepackage{amsfonts}
\usepackage{amssymb}
\usepackage[english]{babel}
\usepackage{graphicx}
\usepackage{float}
\usepackage[titletoc]{appendix}
\usepackage{color}
\usepackage[ulem=normalem]{changes} 

\definechangesauthor[name={Matthias Gerdts},color=red]{MG} 
\definechangesauthor[name={Robert Baier},color=red]{RB} 
\definechangesauthor[name={Robert (old)},color=purple]{Robert}     
\definechangesauthor[name={Olivier (old)},color=purple]{Olivier}   
\definechangesauthor[name={Olivier Bokanowski},color=blue]{OB}

\newcommand{\figs}{.}

\theoremstyle{plain}
\newtheorem{theorem}{Theorem}[section]
\newtheorem{lemma}[theorem]{Lemma}
\newtheorem{prop}[theorem]{Proposition}
\newtheorem{corol}[theorem]{Corollary}
\newtheorem{definition}[theorem]{Definition}

\newtheorem{remark}[theorem]{Remark}

\newcommand{\Z}{{\mathbb Z}}

\newcommand{\R}{{\mathbb R}}

\newcommand{\barr}[1]{\begin{array}{#1}}        
\newcommand{\earr}{\end{array}}			

\newcommand{\MAT}[2]{\left(\barr{#1} #2 \earr \right)}

\newcommand{\mt}{\theta}
\newcommand{\mO}{\Omega}

\newcommand{\be}{\begin{eqnarray}}
\newcommand{\ee}{\end{eqnarray}}
\newcommand{\beno}{\begin{eqnarray*}}
\newcommand{\eeno}{\end{eqnarray*}}

\newcommand{\equivalent}{\Leftrightarrow}
\newcommand{\conv}{\rightarrow}

\newcommand{\cK}{{\mathcal K}}
\newcommand{\cO}{{\mathcal O}}
\newcommand{\cT}{{\mathcal T}}
\newcommand{\cR}{{\mathcal R}}
\newcommand{\cU}{{\mathcal U}}

\newcommand{\cH}{{\mathcal H}}
\newcommand{\cV}{{\mathcal V}}

\newcommand{\Dt}{{\Delta t}}
\newcommand{\Dx}{{\Delta x}}

\newcommand{\MODIF}[1]  {#1} 

\newcommand{\vv}{\MODIF{\vartheta}}            

\definecolor{OliveGreen}{cmyk}{0.64,0,0.95,0.40}

\newif\ifUseBibTeX  
\UseBibTeXtrue

\newcommand{\eps}{\varepsilon}

\begin{document}

\title[Computation of Avoidance Regions for Driver Assistance Systems]
{Computation of Avoidance Regions for Driver Assistance Systems by Using a Hamilton-Jacobi
Approach}

\author[I. Xausa, R. Baier, O. Bokanowski, M. Gerdts]
{
Ilaria Xausa
\address{%
Institut fuer Mathematik und Rechneranwendung,
 Fakultaet fuer Luft- und Raumfahrttechnik,
 Universitaet der Bundeswehr Muenchen,
 D-85577 Neubiberg/Muenchen, Germany.
Email: \texttt{ilariaxausa@gmail.com}}%
\and
Robert Baier
\address{%
Chair of Applied Mathematics, University of Bayreuth, 95440 Bayreuth, Germany.
Email: \texttt{robert.baier@uni-bayreuth.de}}%
\and
Olivier Bokanowski
\address{%
Universit{\'e} de Paris, Laboratoire Jacques-Louis Lions (LJLL), F-75005 Paris, France,
{\&}
Sorbonne Universit{\'e}, LJLL, F-75005 Paris, France
Email: \texttt{olivier.bokanowski@math.univ-paris-diderot.fr}}%
\and
Matthias Gerdts
\address{%
Institute of Applied Mathematics and Scientific Computing, 
Department of Aerospace Engineering, Bundeswehr University Munich, 
Werner-Heisenberg-Weg 39, 85577 Neubiberg, Germany.
Email: \texttt{matthias.gerdts@unibw.de}}%
}
\thanks{This work was partially supported by the EU under the 
7th Framework Programme Marie Curie Initial Training Network
``FP7-PEOPLE-2010-ITN'', SADCO project, GA number 264735-SADCO}


\date{
      \today 
}
\maketitle
\renewcommand{\thefootnote}{\arabic{footnote}}
\pagestyle{myheadings}
\thispagestyle{plain}



\begin{abstract}
We consider the problem of computing safety regions, modeled as nonconvex backward reachable sets, for
a nonlinear car collision avoidance model with time-dependent obstacles.
The Hamilton-Jacobi-Bellman framework is used.
A new formulation of level set functions for obstacle avoidance
is given and sufficient conditions for granting the obstacle
avoidance on the whole time interval are obtained, even though the conditions are checked only at discrete times.
Different scenarios including various road configurations, different geometry of vehicle and obstacles, as well as fixed or moving obstacles, 
are then studied and computed. 
Computations involve solving nonlinear partial differential equations
of up to five space dimensions plus time with nonsmooth obstacle representations,
and an efficient solver is used to this end. 
A comparison with a direct optimal control approach 
is also done for one of the examples.
\end{abstract}


\medskip

\keywords{\small {\bf Keywords:} collision avoidance, Hamilton-Jacobi-Bellman equations, backward reachable sets, level set approach,
high dimensional partial differential equations.

\section{Motivation}

The paper investigates optimal control approaches for the detection of potential collisions in car motions.
The aim is to identify and to compute safety regions for the car by means of a reachable set analysis.
Related techniques for collision detection in real time have been developed in \cite{Althoff2010a} (see also reference therein)
using reachability analysis with zonotopes and linearized dynamics. Reachable set approximations 
have also been obtained through zonotopes in \cite{leGuernic2009}, \cite{Althoff2011f},
\MODIF{\cite{Althoff2015a} (CORA matlab toolbox)},
or through polytopes in \cite{leGuernic2009}, \cite{fal-et-al-2012}.
For an overview of methods see also \cite{bok-for-zid-2010-2,bai-ger-xau-2013}. The new contributions in this paper 
are the ability to approximate nonconvex reachable sets for general nonlinear dynamics accurately while taking into 
account complicated scenarios with obstacles.
The first aim of the paper is to present a verification tool 
for safety systems, illustrated on a car avoidance model,
compare \cite{Nielsson2014,IlariaPHD}. 
The proposed method has the advantage to be capable of approximating the reachable 
set without relying on convex overestimation 
or underestimation. It also avoids a linearization of the nonlinear dynamics.

It is not primarily intended to use this techniques in real time, however, 
a potential approach towards real time computations would be to create
a database of solutions for different scenarios and to use some online interpolation techniques.

It is well known that Hamilton-Jacobi (HJ) approach can be used to modelize and to compute reachable sets~\cite{mit-bay-tom-2005}.
A general setting for taking state constraints (or obstacles) into account for the HJ setting is given in \cite{bok-for-zid-2010-2}, 
the approach we use in the present work.
In \cite{desilles2012collision} the collision avoidance of an unmanned aerial vehicle} (UAV) is modelized as an
infinite horizon control problem with obstacle \MODIF{avoidance} and also solved by using an HJ approach.

For a general presentation of HJ equations we refer to the textbook~\cite{bar-cap-BOOK}. 
A large panel of approximation methods are furthermore available.
Finite difference methods for the approximation 
of nonlinear HJ equations were first proposed in \cite{Crandall-Lions-84}.
Precise numerical schemes (finite-difference type schemes) were later on developed in \cite{osh-shu-91}.
We refer to \cite{SHU-SURVEY-DG-for-HJ} and \cite{fal-fer-BOOK} for recent surveys on related high-order discretization methods. 

In this work, 
we also apply the HJ framework in order to deal with time dependent state constraints as detailed in~\cite{bok-zid-ifac-2011}.

The motion of the car is modeled by the "point mass" model, which is a nonlinear $4$-state model 
(see section~\ref{subsec:4d_model})
with two controls for acceleration/deceleration and for the yaw rate.
More precise car model exists (as the one in \cite{Ger06b}). 
However, the 4-state model is often used for reference in computations and is more easy
to handle numerically by the Hamilton-Jacobi approach.


Once an obstacle
has been detected by suitable sensors (e.g.~radar, lidar), 
our approach can be used to decide whether a collision is going to happen or not. 

In order to \MODIF{compute} the reachable sets, the ROC-HJ solver \cite{rochj} will be used, 
considering several different car models and scenarios.
The software solves Hamilton-Jacobi Bellman (HJB) equations and can be used for \MODIF{computing} the
reachable sets as well as optimal trajectories.
Moreover, for one scenario, we shall also verify our simulations by the DFOG method 
of \cite{bai-ger-xau-2013} (using the OCPID-DAE1 solver \cite{ocode}).
\MODIF{For solving the HJB equations, a numerical method using a mesh grid is introduced,
so numerical errors may occur.
However, as the mesh grid is refined, the numerical error will converges to zero.
Such errors are analyzed in the first example of the numerical section (see Table~\ref{tab:1}). 
}


The plan of the paper is the following.
In Section~2, we consider the 4-dimensional 
point mass model for a car and we describe the problem of the backward reachable set computation 
under different types of nonsmooth state constraints, 
\MODIF{and also introduce level set functions to represent the reachable sets.}
In Section~3, the HJB approach is briefly recalled in our setting.
In Section~4, a general way to construct explicit level set functions associated to state constrains 
is introduced, and a new procedure with respect to collision avoidance is described in more details, 
\MODIF{with particular focus also for the avoidance of rectangular vehicles in the HJB framework.}
Section~5 contains several numerical examples for collision avoidance scenarios, showing the relevance of our approach.
Finally a conclusion is made in Section 6, where we also outline some ongoing works using the HJB approach for collision avoidance.

\section{Problem setting and modelling}

\subsection{Presentation of the problem}
The main tasks in collision avoidance are
to reliably indicate future collisions and -- if possible -- to provide escape 
trajectories if such exist. 
In particular once an obstacle has been detected by suitable sensors (e.g.~radar, lidar), 
we want to be able to decide whether a collision is 
going to happen or not. 

Let $U$ be a nonempty compact subset of $\R^m$ with $m\geq 1$. Let $n\geq 1$ and 
$f:\R^n\times U\conv \R^n$ be Lipschitz continuous with respect to $(z,u)$.
Let
$$\cU:=\{u:[0,\infty)\conv U \ |\ \mbox{$u$ measurable}\}$$
be the set of control policies.
Given an initial state $z_0\in\R^n$, 
we denote by $z_{z_0}^u$
the \emph{absolutely continuous}
solution of the following nonlinear dynamical system
\be\label{eq:dynamics}
  & & \dot z(s)=f(z(s),u(s)) \quad \mbox{for }\ s\geq 0,\\
  & & z(0)=z_0.
\ee
\MODIF{(More precisely, \eqref{eq:dynamics} stands only for "almost every" $s\geq 0$ since the control 
$u(\cdot)$ is only a measurable function of the time.)}
Let $\Omega$ and $(\cK_s)_{s\geq 0}$ be nonempty closed sets of $\mathbb{R}^n$.

\begin{definition}\label{def:reachset}
The \emph{backward reachable set} 
associated to the dynamics $f$,
for reaching the target $\Omega$ within time $t$ and respecting the state constraints $(\cK_s)_{s\geq 0}$
is defined as:
\beno
  & \cR^{f}_{\Omega,(\cK_s)_{s}}(t) := 
  \bigg\{z_0\in\R^n \ |\  \exists u\in L^\infty((0,\infty),U):\ z_{z_0}^u(t)\in\Omega \\
  & \hspace{2cm} \mbox{and\ } z_{z_0}^u(s)\in\cK_s\ \mbox{for\ }s \in [0,t]\bigg\}. 
\eeno
If the context is clear, we use the abbreviation $\cR(t)$.
\end{definition}
The definition corresponds to the set of points 
$z_0$ such that there exists trajectory starting from $z_0$ ending inside the target area $\Omega$ 
at the given time $t\geq 0$, while satisfying the
state constraints.
In the application in the next subsection, $(\cK_s)_s$ will model that the car staying on the road and avoiding the obstacles.

\subsection{The 4-dimensional point mass model}
\label{subsec:4d_model}
The \emph{point mass model} used hereafter is a simplified four-dimensional nonlinear state model for the car,
where the controls are the acceleration/deceleration and yaw rate. 


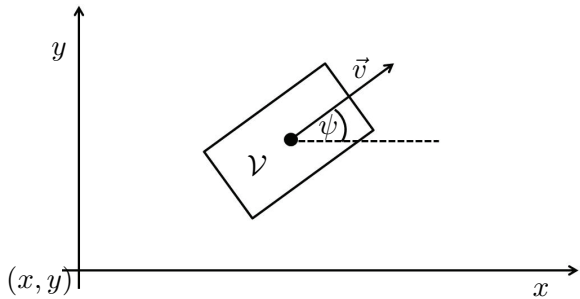
\begin{figure}
\begin{center}
  \includegraphics[angle=0,width=7cm]{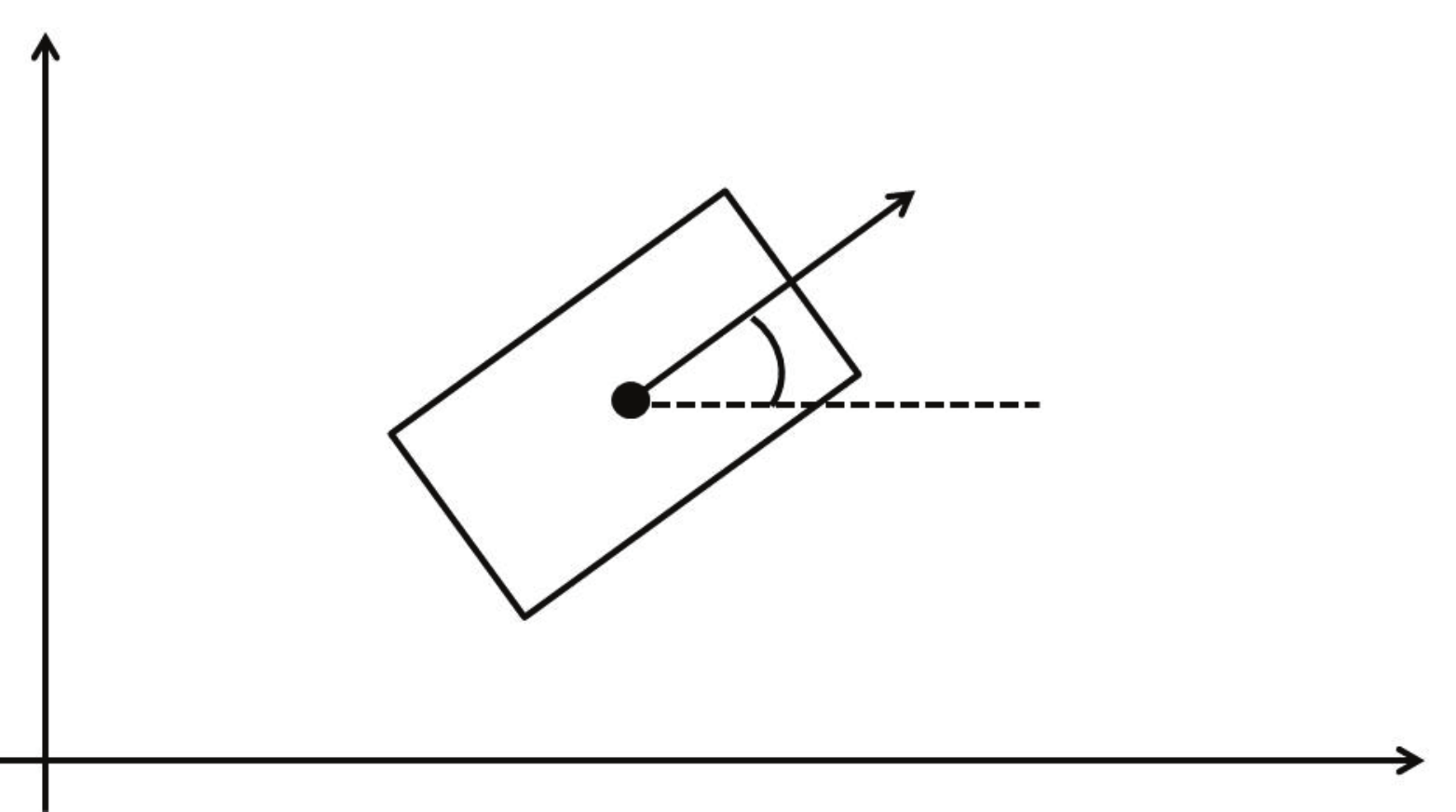}
\setlength{\unitlength}{1cm}
\begin{picture}(0,0)(0,0)
\put(-3.85,2.05){$\psi$}
\put(-4.8,1.5){$\mathcal V$}
\put(-3.4,2.8){$\vec{v}$}
  \put(-1,-0.1){$x$}
  \put(-7.4,3.1){$y$}
  \put(-8,0){$(x,y)$}
\end{picture}
\caption{\label{fig:4d_model} vehicle $\mathcal V$}
\end{center}
\end{figure}

The center of gravity of the car is identified with its coordinates $(x(t),y(t))$, 
$v(t)$ denotes the module of the velocity of the car and $\psi(t)$ is the yaw angle, 
see Fig.~\ref{fig:4d_model}.  
The equations of motion are then given by
\begin{subequations}\label{eq:model} 
\be
  x'   & = & v\cos(\psi)\\
  y'   & = & v\sin(\psi)\\
  \psi'& = & w\\
  v'   & = & a
\ee
\end{subequations}
where $a(t) \in [a_{min},a_{max}]$ is the control for acceleration (if $a(t) > 0$) or deceleration (if $a(t) < 0$), 
and $w(t)$ is the yaw rate, 
such that $|w(t)| \leq w_{max}$.  
 
\subsection{Level set functions for target and state constraints}
The \emph{target set} $\Omega$ will represent the target area. 
In our examples we will consider a region in the form of
\be\label{eq:omega}
  \Omega=\bigg\{z=(x,y,\psi,v)\in\mathbb{R}^4\ \mid\ x \geq x_{target}, y \geq y_{target}, |\psi|\leq \eps\bigg\}
\ee
for a small threshold $\eps>0$. 
The conditions on $x$ and $y$ realizes that the vehicle will reach a secure end position after the maneuver, while the conditions on
the yaw angle assures that the car is driving approximately horizontally so that
the car is not heading to leave the road after the maneuver.

An important tool, which will be used in this paper, is the notion of \emph{level set functions} associated to a given set.
For the target constraint $\mO$, we will associate a Lipschitz continuous function $\varphi:\R^n\rightarrow\R$ (here with $n=4$),
such that
$$ 
  z\in \mO \quad \Longleftrightarrow \quad \varphi(z)\leq 0.
$$
Indeed there always exists such a function, since one can use the \emph{signed distance} to $\mO$ 
($\varphi(z)=d(z,\mO)$ if $z\notin \mO$, and $\varphi(z)=-d(z,\R^n\backslash\mO)$ if $z\in \mO$).
For instance, in the case of \eqref{eq:omega}, 
we can consider:
$$ 
  \varphi(z) := \max\bigg(-(x-x_{target}),\ -(y-y_{target}),\ |\psi|-\eps\bigg).
$$
More complex level set functions will be constructed in Section \ref{sec:4}
but they may no more correspond to usual distance functions.

\medskip

Now for the case of the state constraints, different 
road geometries will be considered, as well as different fixed and/or moving obstacles that the car must avoid.
In all cases, we will manage to model all state constraints
by a simple, explicit, level set function $g$. 

In the particular case of fixed configurations, state constraints are of the form $z(s)\in \cK$ for all $s\in[0,t]$
(for a given closed set $\cK$ of $\R^n$),
and we will construct a Lipschitz continuous function $g:\R^n\rightarrow\R$ such that 
$$ 
  z\in \cK \quad \Longleftrightarrow \quad g(z)\leq 0.
$$

In the general case of moving obstacles, the state constraints are of the form
$$
  z(s)\in \cK_s, \quad \forall s\in[0,t],
$$
where $(\cK_s)_{s\geq 0}$ is a family of closed sets.
Then we need to introduce a time-dependent level set function (still denoted $g$)
$g:\R^n\times \R \rightarrow \R$, such that, for any $s\geq 0$,
$$
  z\in \cK_s \quad \Longleftrightarrow \quad g(z,s)\leq 0.
$$
The obstacles and corresponding level set functions will be more precisely described and constructed in Section~\ref{sec:4}. 

\section{Hamilton-Jacobi-Bellman approach}
In order to apply the HJB approach of \cite{bok-zid-ifac-2011}, assumptions (H1)-(H3)
will be assumed on the objects modelling the car traffic scenario as defined in the previous section. 
Assumption (H1) is natural for controlled systems and fulfilled for our model.
Assumption (H2) and (H3) allow for nonsmooth representation of targets and state constraints.

\begin{itemize}
\item[(H1)]\label{h1}
 $f:\mathbb{R}^n\times U\rightarrow\mathbb{R}^n$ is a continuous function and Lipschitz continuous in $z$
 uniformly in $u$, 
 i.e., $\exists L\geq0, \forall z_1,z_2$, $\forall u\in U$:  $|f(z_1,u)-f(z_2,u)| \leq L |z_1-z_2|$.

\item[(H2)]\label{h3} 
$\mO$ is a nonempty closed set of $\R^n$.
Let $\varphi:\R^n\conv\R$ be Lipschitz continuous
and a level set function for the target, i.e.,
$$\mbox{$\varphi(z)\leq 0$ $\equivalent$ $z\in \mO$.}$$
\item[(H3)]\label{h4} 
 $(\cK_s)_{s\in[0,T]}$ are a family of subsets of $\R^n$ 
 such that there exists a Lipschitz continuous level set function $g:\R^n\times\R\conv \R$ with
 $$
    \mbox{ $g(z,s)\leq 0 $ $\equivalent$ $z\in \cK_s$ 
           \quad for all $s\in[0,T]$, $z\in\R^n$.
    } 
 $$
\end{itemize} 
The existence of a function $g$ with the Lipschitz regularity is discussed in \cite{bok-zid-ifac-2011}.

Furthermore, we will assume the following convexity assumption on the dynamics: 
\begin{itemize}
\item[(H4)]
 For all $z\in \R^n$ the velocity set $f(z,U)$ is convex.
\end{itemize}
In our case, for the dynamics given in \eqref{eq:model} (resp.~\eqref{eq:ext_F})
this convexity assumption is satisfied since the map $u \mapsto f(x,u)$ is affine. It allows to have the compactness of the set of 
trajectories and therefore the existence of a minimum for the value function \eqref{eq:value} (or \eqref{eq:value2}).


However this condition is not mandatory and we could use the Filippov-Wa\'zewski 
relaxation theorem 
(see~\cite[Sec.~2.4, Theorem~2]{aub-cel-BOOK})
to work with a convexified dynamics in the case $f(x,U)$ is nonconvex.



We now focus on backward reachable sets
$\cR^{{f}}_{\Omega,(\cK_s)}(t)$ associated to a dynamics $f$
and how to compute such reachable sets. 


\begin{remark}
Assume that there is no time dependency in the dynamics, and that 
$\cK_s=\cK$ for all $s \geq 0$.
Let the {\em capture basin}, or backward reachable set until time $t$,
be defined by 
\be \label{eq:brs}
  Cap^{{f}}_{\Omega,(\cK_s)_{s}}(t) := 
  \bigg\{z_0\in\R^n\,\vert\, \exists \tau \in[0,t],\ \exists u\in \cU,\ z_{z_0}^u(\tau)\in\Omega\nonumber\\
     \ \mbox{and}\ z_{z_0}^u(s)\in\cK_s \ \mbox{for all $s \in [0,\tau]$}\, 
  \bigg\}
  \label{eq:capt_basin2}
\ee
(following~\cite[Subsec.~1.2.1.2]{ABS-P}) so that we consider trajectories that can reach the target until time $t$.
We will also use the abbreviation $Cap(t)$.
Then the set~\eqref{eq:brs} is also a reachable set for a modified dynamics $\hat f$:
$$
  Cap^{f}_{\Omega,\cK}(t) = \cR^{\hat f}_{\Omega,(\cK_s)}(t) 
$$
where $\tilde f(z_0,\hat u):=\lambda f(z_0,u)$ for $\hat u =(u,\lambda)\in\hat U=U  \times [0,1]$
(see for instance \cite{mit-bay-tom-2005}). 
Here, a new virtual control $\lambda(\cdot)$ with $\lambda(s) \in [0,1]$ is introduced.
\end{remark}

We first consider the case of time-invariant state constraints (such as the road and non-moving obstacles), i.e.,
$$
   \cK_s:=\cK\quad \forall s \geq 0.
$$
Let us associate to this problem the following {\em value function} $\vv$:
\be\label{eq:value}
  \vv(z_0,t):=\inf_{u\in\cU} \  \max\left( \varphi(z^u_{z_0}(t)),\max_{s\in[0,t]}g(z^u_{z_0}(s)) \right).
\ee
Here, we simply denote $g(z,s)\equiv g(z)$.
Such a value function involving a supremum cost has been studied by Barron and Ishii in~\cite{Barron_Ishii89}.
As a consequence of \cite{bok-for-zid-2010-2}, we have:

\begin{prop}
The value function $\vv$ is a level set function for 
$\cR^f_{\Omega,(\cK_s)_{s}}(t)$
in the sense that the following holds:
\be\label{eq:captureandlevel}
  \cR^f_{\Omega,(\cK_s)_{s}}(t)
  =\{z_0 \in \R^n \,\vert\, \vv(z_0,t)\leq0\}
\ee
\end{prop}
In particular assumptions (H2) and (H3) are essential for \eqref{eq:captureandlevel} to hold.
Furthermore,  $v$ is the unique continuous \emph{viscosity solution} (in the sense of \cite[Sec.~I.3]{bar-cap-BOOK})
of the following Hamilton-Jacobi (HJ) equation
\begin{subequations}\label{eq:hjb-0}
\be
  \min\big(\partial_t \vv+\cH(z,\nabla_z \vv),\  \vv-g(z)\big)=0, & \quad &  t>0,\ z\in\R^n,\\   
  \vv(z,0)=\max(\varphi(z),g(z)), & \quad & z\in \mathbb{R}^{n},
\ee
\end{subequations}
where 
$$  \cH(z,p):=\max_{u\in U} (-f(z,u) \cdot p), \quad p \in \R^m,
$$
is the \emph{Hamiltonian}.

\MODIF{
\begin{remark}
  From the point of view of the vehicle obstacle avoidance problem, the value function $\vv$ has the following 
  property :
  $\vv((x_0,y_0,\psi_0,v_0),t)\leq 0$  if and only if it is possible to drive the vehicle from the initial
  position $z_0=(x_0,y_0,\psi_0,v_0)$ with dynamics~\eqref{eq:dynamics}, towards the target set $\mO$, 
  avoiding the obstacles (such as possible other obstacle vehicles) before time $t$.
\end{remark}
}

We now consider the computation of backward reachable sets for {\em time dependent state constraints}, and follow the approach of~\cite{bok-zid-ifac-2011}.
We consider~\eqref{eq:model} with the new state variable $\xi:=(z,\tau)$ and the "augmented" dynamics with values in $\R^5$:
\be
  \label{eq:ext_F}
  F(\xi,u):=\MAT{c}{f(z,u)\\ 1}.
\ee
This means, in solving the corresponding PDE, we have to consider an additional state variable:
the dynamics is augmented by the equation
\be
  \tau'& = & 1.
\ee

Let also be $\xi_0:=(z_0, 0)$
and trajectories $\xi_{\tau_0,\xi_0}^u$  associated to $F$, 
fulfilling
$$
   \dot\xi(s)=F(\xi(s),u(s)) \ \mbox{and}\ \xi(\tau_0)=\xi_0.
$$
For a fixed $T>0$, let 
\be
  & & \tilde{\Omega}:=\bigcup_{s\in [0,T]} \Omega\times\{s\} \ \equiv \  \mO\times[0,T]
  \label{eq:ext_target}
\ee
and
\be
  & & \tilde{\cK}:= \bigcup_{s\in[0,T]} \cK_s\times\{s\}.
  \label{eq:ext_state_constr}
\ee

Then it holds:
\begin{prop}
For all $t\geq0$, we have:
$$\forall s\in[0,t],\ 
  z_0\in \cR^f_{\Omega,(\cK_s)_{s}}(t) \iff (z_0,0) \in \cR^F_{\tilde{\Omega},\tilde{\cK}}(t).
$$
\end{prop}
We extend the definition of $\varphi$ by $\varphi(z,s):=\varphi(z)$ 
so that for any $\xi_0\in\R^n\times \R$ and $\tau\geq 0$ 
we can associate a value $w$ as follows:
\be\label{eq:value2}
  w(\xi_0,\tau):=\inf_{u\in\cU} \  \max\left( \varphi(\xi^u_{\xi_0}(\tau)),\max_{s\in[0,\tau]} g(\xi^u_{\xi_0}(s)) \right)
\ee
Then, one can verify that $w$ is Lipschitz continuous and the following theorem holds: 

\begin{theorem}
Assume (H1)-(H4) and consider $w$ from \eqref{eq:value2}.\\
$(i)$ For every $\tau \geq0$ we have:
$$  
  \cR^f_{\Omega,(\cK_{s})_{s}}(\tau)=\left\{z\in\R^n,\ w((z,0),\tau)\leq 0 \right\}.
$$
$(ii)$ $w$ is the unique continuous viscosity solution of 
\begin{subequations}\label{hjb-t}
\be
 & &  \min(\partial_\tau w+\cH((z,t),(\nabla_z w, \partial_t w)),\ w((z,t),\tau)-g(z))=0, \nonumber \\
 & & \hspace{6cm} \quad \tau>0,\ (z,t)\in \R^{n+1}, \\
 & & w((z,t),0)=\max(\varphi(z),g(z,t)),\quad (z,t)\in \R^{n+1}.
\ee
\end{subequations}
where for any $\xi=(z,t)$ and $(p_z,p_t)\in\R^n\times\R$:
$$
  \cH((z,t),(p_z,p_t)):=\max_{u\in U} (-f(z,u)\cdot p_z - p_t).
$$
\end{theorem}

Once the backward reachable set is characterized by a viscosity solution of~\eqref{hjb-t}, 
it is possible to use a pde solver to find the solution on a grid.
We refer to \cite{bok-for-zid-2010-2} for the specific approximation of \eqref{eq:hjb-0}
(resp.~\eqref{hjb-t}).

%

\medskip

\paragraph{\bf Minimal time function and optimal trajectory reconstruction.}
In the case of fixed state constraints (i.e., $\cK_s\equiv \cK$),
the \emph{minimal time function}, denoted by $\cT$, is defined by:
$$
  \cT(z_0):=
   \inf \{t\in[0,T] \,|\,  \exists u\in\cU: z_{z_0}^{u}(t)\in\Omega \mbox{\ and\ }z_{z_0}^{u}(s)\in\cK\ 
          \mbox{for all $s\in [0,t]$}  \},
$$
and if no such time $t$ exists then we set $\cT(z_0)=\infty$.
Let $\vv$ be defined as in~\eqref{eq:value}.
It is easy to see that the function satisfies
\be\label{eq:Tmin}
  \cT(z_0)=\inf\{t\in[0,T],\ \vv(z_0,t)\leq 0\}.
\ee
\MODIF{In particular, if $\cT(z_0)=+\infty$, then there is no $t\in[0,T]$ such that $\vv(z_0,t)\leq0$.}
Notice that $\cT$ can be discontinuous even though $\vv$ is always Lipschitz continuous.
No controllability assumptions are used in the present approach.

The optimal trajectory reconstruction is then obtained by minimizing the minimal time function along possible trajectories,
see for instance Falcone~\cite{falcone}, Soravia~\cite{soravia}.  
From the numerical point of view this allows to store only the minimal time function and 
not the value function which would have one more variable (the time variable).
(Note that precise trajectory reconstruction results from the value function can be found in \cite{row-vin-91,ass-bok-des-zid-15}.)
The minimal time function can be shown to satisfy a dynamic programming principle of the 
form:
\be\label{eq:T-dpp}
  \inf_{u(.)}\cT(z^u_{\xi}(h)) = \cT(\xi)-h, \quad
  \text{for all $0<h\leq\cT(\xi)<\infty$,}
\ee
where the infimum is over all control functions $u(.)\in \cU$ such that the trajectory $z^u_\xi(.)$ satisfy the state constraints.

More precisely, let $h>0$ be a given time step, let $N\geq 1$ be some integer (a maximal number of iterations). 
Assume that the starting point $z_0$ is satisfying $\cT(z_0)<\infty$,
and that we aim to reach $z_n:=z^u(t_n)\in\Omega$
at some future time $t_n=nh>0$, with $n\leq N$. This is equivalent to require $\cT(z_n)=0$. 
For a given small threshold $\eta>0$ and for a given control discretization of the set 
$U$, say $(u_k)_{k=1,\dots,N_u}\subset U$,
we consider the following iterative procedure:
\beno
  & & \hspace{-0.5cm} \mbox{while $n<N$ and $\cT(z_n)\geq \eta$ do:}\\
  & & \qquad \mbox{Find}\  u_{k^*} := {\mathrm{argmin}}_{u_k} \cT\big( \bar z^{u_k}_{z_n}(h) \big)\\
  & & \qquad \mbox{Set}\  z_{n+1}:= \bar z^{u_{k^*}}_{z_n}(h)\\
  & & \qquad \mbox{if}\ \cT(z_{n+1})=\infty \ \mbox{stop}, \mbox{otherwise set $n\leftarrow n+1$}.
\eeno
\MODIF{More precisely,} the procedure is stopped if $\cT(z_{n+1})=+\infty$ (we cannot reach the target from this point),
if $\cT(z_{n+1})<\eta$ (target reached up to the threshold $\eta$), or if the maximal number of iterations is reached.
In the algorithm, $\bar z^{u_k}_{z_n}(h)$ denotes a \emph{one-step second-order Runge-Kutta approximation} 
of the trajectory with fixed control $u_k$ on $[t_n,t_{n+1}[$. For instance,
the \emph{Heun scheme with piecewise constant selections} uses the iteration:
$$
   \bar z^{u_k}_{z_n}(h):= z_n + \frac{h}{2} (f(z_n,u_k) + f(z_n + h f(z_n,u_k),u_k)).
$$
It is also possible to do smaller time steps between $[t_n,t_n+h]$ in order to improve the precision for a given control $u_k$. 
Nevertheless, numerical observations show that the approximation 
is in general more sensitive to the control discretization $(u_k)_{k=1,\ldots,N_u}$ of the set $U$.

In the time-dependent case this minimal time function can be defined in a similar way from the value $w$ (we refer 
to \cite{bok-zid-ifac-2011} for details), and the optimal trajectory reconstruction follows the same lines 
with the "augmented" dynamics~\eqref{eq:ext_F}.

\begin{remark}\label{rem:3.4}
Notice that the discretization in the optimal trajectory reconstruction can be done completely independently
from the discretization method used for solving the HJ equations as well as the minimal time function.
\end{remark}

When reconstructing the trajectory with the minimal time function 
in the iteration above, 
a simple yet important property is:
\begin{lemma}\label{lem:3.4}
Assume that $\cT(z_n)<+\infty$. Then $g(z_n)\leq 0$.
This means in particular that the state $z_n$ satisfies the state constraints.
\end{lemma}
\begin{proof}
If $\tau_n:=\cT(z_n)<\infty$ then $\vv(z_n,\tau_n)\leq 0$. On the other hand by definition of $\vv$: $g(z_n)\leq \vv(z_n,\tau_n)$.
This concludes to the desired result.
\end{proof}

In particular, if the trajectory reconstruction satisfies $\cT(z_n)<\infty$, for $n=0,\dots,N$
(and this is expected since $\cT(z_0)<+\infty$ and in view of the dynamic programming principle \eqref{eq:T-dpp}),
then we will have
\be
    g(z_n)\leq 0, \quad \forall n=0,\dots,N,
\ee
and therefore all points $(z_n)_{n=0,\dots,N}$ satisfy the state constraints.
In the time-dependent case, the condition $g(z_n)\leq 0$ must be replaced by $g(z_n,\tau_n)\leq 0$.

\section{Level set functions for different problem data and collision avoidance} \label{sec:4}

In this section we construct Lipschitz level functions to represent obstacle, state constraints and targets satisfying assumptions (H2)-(H3).
The aim is also to illustrate how to obtain explicit analytic formula for some specific obstacles and vehicles.
In some cases, the avoidance of two convex set may not be easily characterized by the negativity of some analytic function. 
We will propose a new way to use simplified analytic level set functions in order to obtain a collision avoidance characterization.  

\subsection{Road configurations}
Let us first describe different simple road geometries depicted in Fig.~\ref{fig:roads} 
which will be used in the next numerical section.
The road will be denoted by $\cK_r$, a subset of $\R^2$.



\begin{figure}
\begin{center}
\begin{tabular}{cc}
\hspace{-0.5cm}\includegraphics[angle=0,width=6cm]{\figs/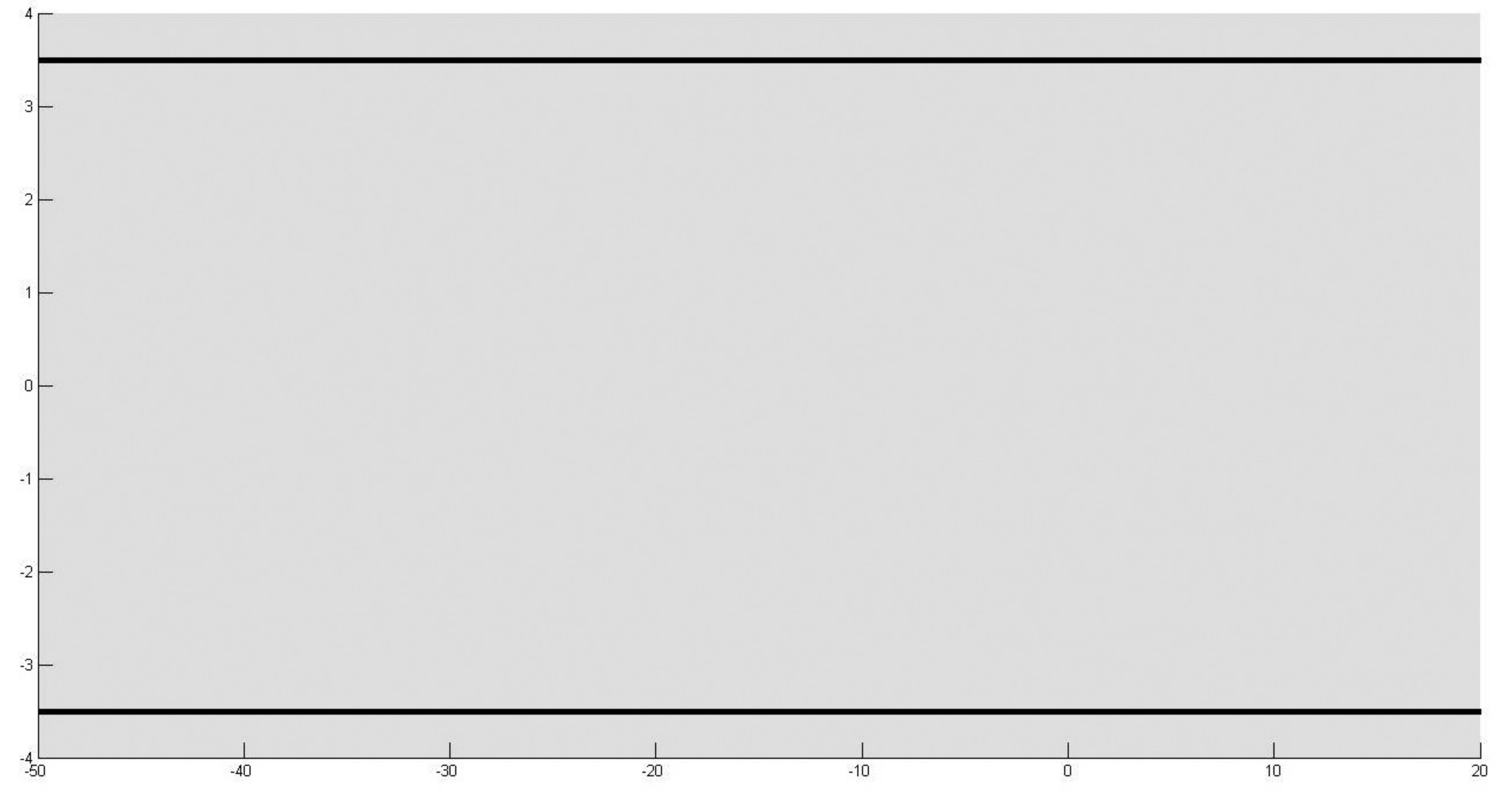} & 
\hspace{-0.5cm}\includegraphics[angle=0,width=6cm]{\figs/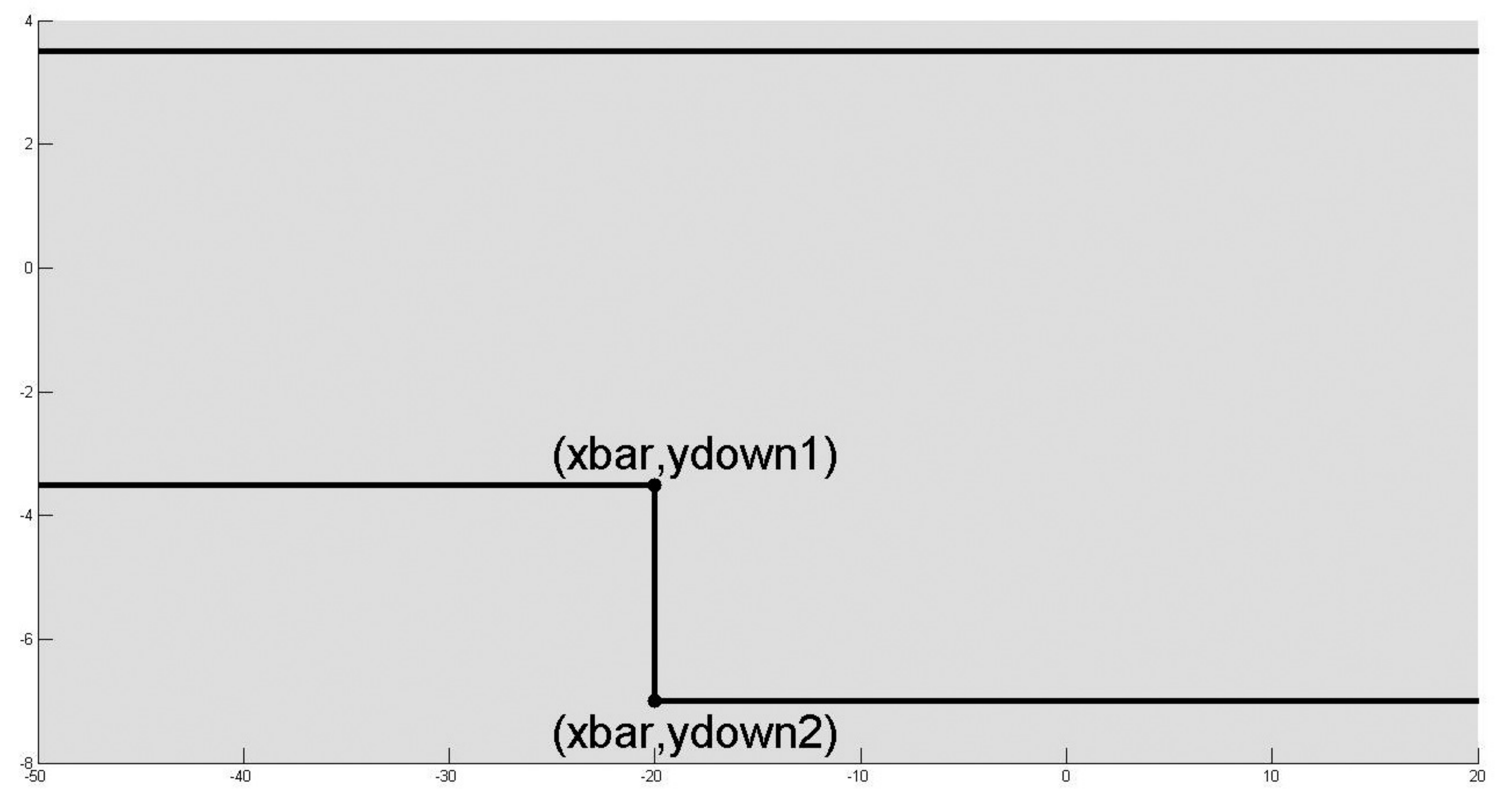}\\
\hspace{-0.5cm}\includegraphics[angle=0,width=6cm]{\figs/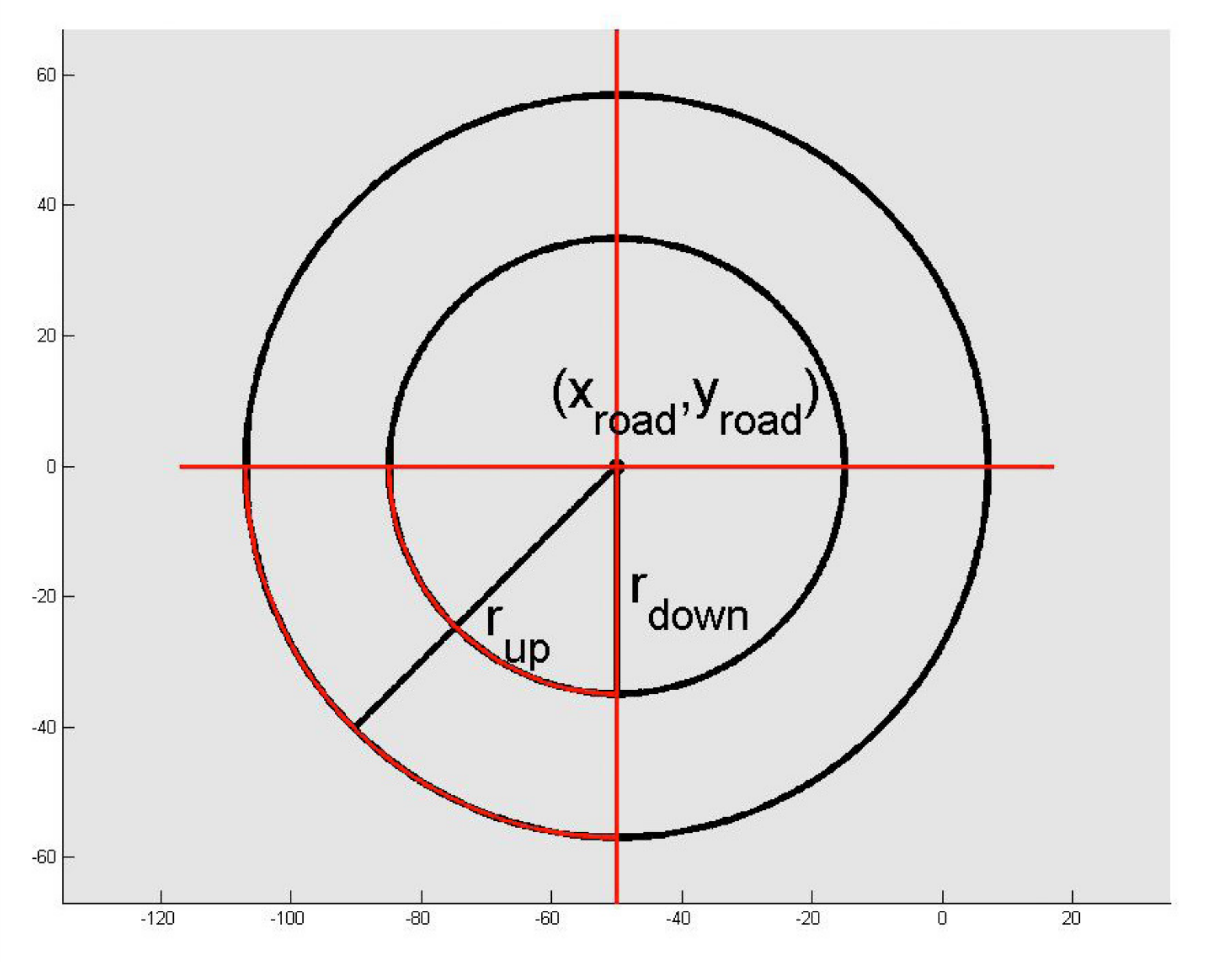} &
\hspace{-0.5cm}\includegraphics[angle=0,width=6cm]{\figs/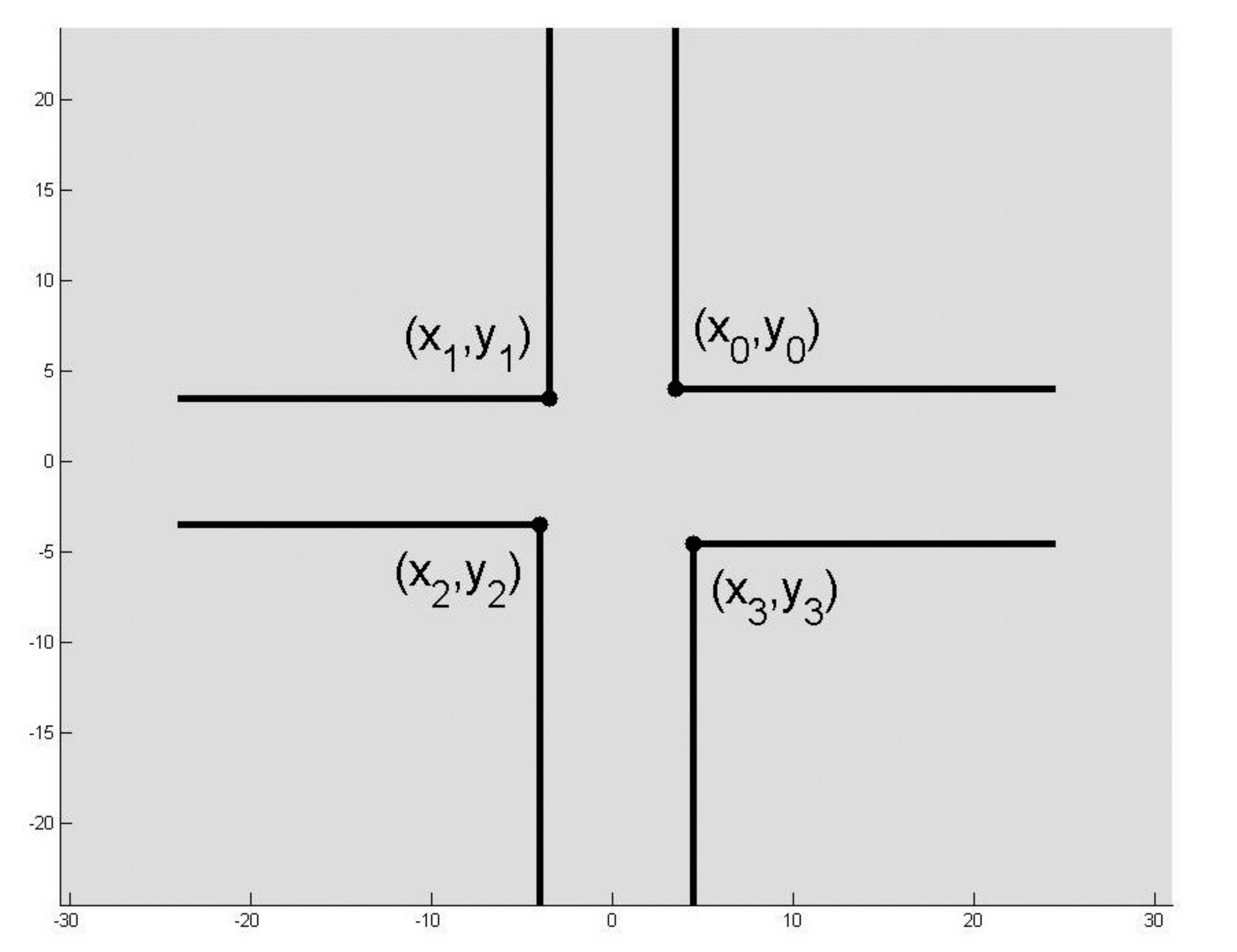}
\end{tabular}
\end{center}
\caption{
\label{fig:roads}
different road geometries: straight (top-left), varying width (top-right), curved (bottom-left), 
  crossing (bottom-right)}
\end{figure}

$\bullet$ 
A straight road with constant width:
\be\label{eq:Kr_straight}
\cK_r=\bigg\{(x,y)\in\R^2 \,|\, y_{down}\leq y \leq y_{up}\bigg\},
\ee
with
$y_{down}=-3.5\mbox{ and } y_{up}=3.5$
(values are typically in meters). A level set function associated to $\cK_r$ can be given by
\be\label{funct_straight}
g(x,y) := \max\bigg(-(y-y_{down}),\ -(y_{up}-y)\bigg).
\ee

$\bullet$ 
A straight road with varying widths, modeled as:
\begin{equation}\label{set_straightlarger}
\cK_r:=\bigg\{(x,y)\in\mathbb R^2\, | \, y_{down}(x)\leq y \leq y_{up}\bigg\},
\end{equation}
where
\begin{equation}\label{set_straightlarger_more}
y_{down}(x)=\left\{\begin{array}{ll}
y_{down_1} & \mbox{ if } x \leq \bar x,\\
y_{down_2} & \mbox{ if } x >  \bar x
  \end{array}
  \right.
\end{equation}
and $y_{up}, y_{down_1}, y_{down_2}, \bar x$ are constants. 
It would be natural to define the level set function of \eqref{set_straightlarger} as in \eqref{funct_straight} 
where instead of $y_{down}$ we use the function $y_{down}(x)$ defined in \eqref{set_straightlarger_more}.
However, \eqref{set_straightlarger_more} leads to a discontinuous function $g$,
which is not convenient for our purposes. In this particular case, 
the following definition is better suited because it is Lipschitz continuous in $(x,y)$:
\be
  g(x,y) & = &  \max \bigg( \min\big(-(y-y_{down_1}), -(x-\bar x)\big),  \nonumber \\
         &   &  \hspace{4cm}  -(y-y_{down_2}), (y-y_{up}) \bigg).  \label{funct_straightlarger}
\ee

$\bullet$ 
The circular curve shown in Figure~\ref{fig:roads} is defined as
\be
  \cK_r & = &  \{ (x,y), \  \mt_{\min}\leq \Theta(x,y)\leq \mt_{max},\   \nonumber\\
   &  & \hspace{3cm}
  r_{down}\leq \rho(x,y) \leq r_{up} \}
  \label{eq:set_curve}
\ee
$\mt_{min}<\mt_{\max}$ are two limiting angles for the road boundaries (with $\mt_{\max} \leq \mt_{\min}+2\pi$),
$(x_{c},y_{c})$ denotes the center of the road circle,
$0<r_{up}<r_{down}$ the radius of the bigger and smaller circle respectively,
\beno
\begin{array}{r@{\,}c@{\,}l@{\,}c@{\,}l@{\,}}
  r   & = & \rho(x,y)   & := & \sqrt{(x-x_c)^2 + (y-y_c)^2} \\
  \mt & = & \Theta(x,y) & := &  \displaystyle \arctan(\frac{y-y_c}{x-x_c}) + k_{x,y}\pi, \ \ \mbox{ with $k_{x,y}\in\Z$ } 
\end{array}
\eeno
denotes a continuous representation on $\cK_r$ such that $x-x_c=r \cos(\mt)$ and $y-y_c=r \sin(\mt)$.
The level set function is defined as 
\be
  g(x,y) & = &  \max \bigg(\rho(x,y)- r_{up},\ -(\rho(x,y)- r_{down}), \nonumber \\
   &  & \hspace{2cm} \Theta(x,y)-\mt_{\max},\ - (\Theta(x,y)-\mt_{\min}) \bigg).
   \label{funct_curve}
\ee



\begin{remark}\label{rem:K1cupK2}
A general and simple rule for constructing Lipschitz continuous level set functions is the following.
Assume that $g_1$ (resp.~$g_2$) are Lipschitz continuous level set functions for the set $\cK_1$ (resp.~$\cK_2$), that is, 
$g_i(x)\leq 0$ $\Leftrightarrow$ $x\in \cK_i$, for $i=1,2$. Then
\be
  \max(g_1(x),g_2(x))\leq 0 & \Leftrightarrow & x\in \cK_1 \cap \cK_2  \label{eq:Kinter}\\
  \min(g_1(x),g_2(x))\leq 0 & \Leftrightarrow & x\in \cK_1 \cup \cK_2. \label{eq:Kunion}
\ee
Hence $\max(g_1,g_2)$ (resp.~$\min(g_1,g_2)$) is Lipschitz continuous and it
can be used as a level set function for $\cK_1\cap \cK_2$ (resp.~$\cK_1\cup\cK_2$).
Then more complex structures can be coded by combining~\eqref{eq:Kinter} and~\eqref{eq:Kunion}.
This is related to well-known techniques for (signed) distance functions
in computational geometry (see e.g.~\cite{fayolle-2008, hart-1996})
\end{remark}

$\bullet$ A crossing with corner points $(x_i,y_i)_{i=0,\dots,3}$: 
Let the upper right part be defined as $\cK_0:=\{ x-x_0\leq 0\ \mbox{or} \ y-y_0\leq 0\}$, 
and similarly, 
$\cK_1:=\{ y-y_1\leq 0 \ \mbox{or}\ -(x-x_1)\leq 0 \}$ (upper left part),
$\cK_2:=\{ -(x-x_2)\leq 0,\ \mbox{or}\ -(y-y_2)\leq 0\}$ (lower left part),
$\cK_3:=\{ -(y-y_3)\leq 0,\ \mbox{or}\ x-x_3\leq 0\}$ (lower right part).

Following Remark~\ref{rem:K1cupK2},
a level set function for $\cK:=\bigcap_{i=0,\dots,3} \cK_i$
can be obtained by 
\be
  g(x,y) & = &  \max\big( \min(x-x_0,y-y_0),\ \min(y-y_1,-(x-x_1)), \nonumber \\
   &   & \hspace{0.5cm} \min(-(x-x_2),-(y-y_2)),\ \min(-(y-y_3),x-x_3) \big), \nonumber \\
   &   & 
  \label{funct_crossing}
\ee

More general ways to construct level set functions for roads delimited by polygonal lines could be obtained following 
similar ideas.
Also let us mention that the modeling of the road boundaries via piecewise cubic polynomials or B-splines can be obtained 
as in~\cite{Ger05}.

\if{
\em
A more general way to construct level set functions for roads delimited by polygonal lines is 
the following.
\begin{center}
\includegraphics[angle=0,width=7cm]{\figs/generalroad-eps-converted-to.pdf}
 \end{center}
We characterize the set describing the road as
\begin{equation}\label{set_generalroad}
\mathcal K_r=\left\{(x,y)\in\mathbb{R}^2\, | \, p_{1}(x,y)\leq 0 \mbox{ and } p_{2}(x,y)\geq0\right\}
\end{equation}
For $j=1,2$, $p_{j}:\mathbb R^2\rightarrow \mathbb R$ are interpolation functions 
of the points $(x_{ji},y_{ji})$, $i=0,\dots,n\}$ detected by the sensor 
as points belonging to the upper ($j=1$) and lower ($j=2$) bounds of the road. We write
\begin{equation}
p_{j}=q_{ji}(x,y) \quad \forall x\in[x_{ji}, x_{ji+1}], y\in[y_{ji}, y_{ji+1}] 
\end{equation}
where $q_{ji}:\mathbb R^2 \rightarrow R$ are the polynomials interpolating the pair of
knots $(x_{ji}, y_{ji})$ and $(x_{ji+1}, y_{ji+1})$. The level set function associated to set \eqref{set_generalroad} is:
\begin{equation}\label{funct_generalroad}
\begin{array}{rll}
g(x,y)=&\multicolumn{2}{l}{\min\limits_{j=1,2} g_{j}(x,y)}\\
g_{j}(x,y)=&\max\limits_{i=0,\dots,n-1}\bigg(
\min\big( &
(x-x_{ji+1}\leq0),\\
&&-(x-x_{ji}\leq0),\\[0.2cm]
&&(y-y_{ji+1}\leq0),\\[0.2cm]
&&-(y-y_{ji}\leq0),\\
&&(-1^{j+1}q_{ji}(x,y)\leq0)
\quad\big)
\quad\bigg)
\end{array}
\end{equation}

\begin{remark}
In the particular case where the functions $p_j$ are piecewise linear and convex, 
the level set function of \eqref{set_generalroad} is
\begin{equation}\label{funct_generalroad_convex}
\begin{array}{rl}
      g:= & \max_j g_{j} \\
  g_{j}:= & \max_{i=0,\dots,n-1} \left(-1^{j+1}q_{ji}\leq0\right).
\end{array}
\end{equation}
\end{remark}

Another option is the modeling of the road via piecewise defined cubic polynomials
or B-splines as in~\cite{Ger05}.
}\fi

\subsection{Obstacles and corresponding level set functions}
Next, $k\geq 1$  additional obstacles 
are considered (obstacles which are different from the road).
There are modeled  as disks or rectangles that the car has to avoid (the car being itself
modeled in the form of a disk or a rectangle), therefore defining another type of state constraint.


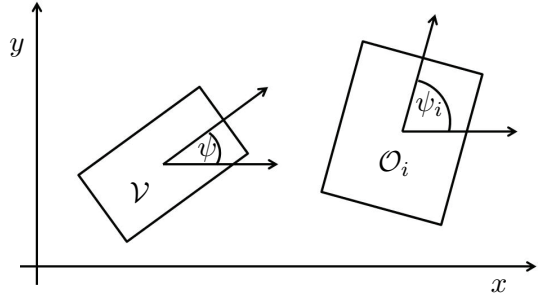
\begin{figure}
\begin{center}
\includegraphics[angle=0,width=7cm]{\figs/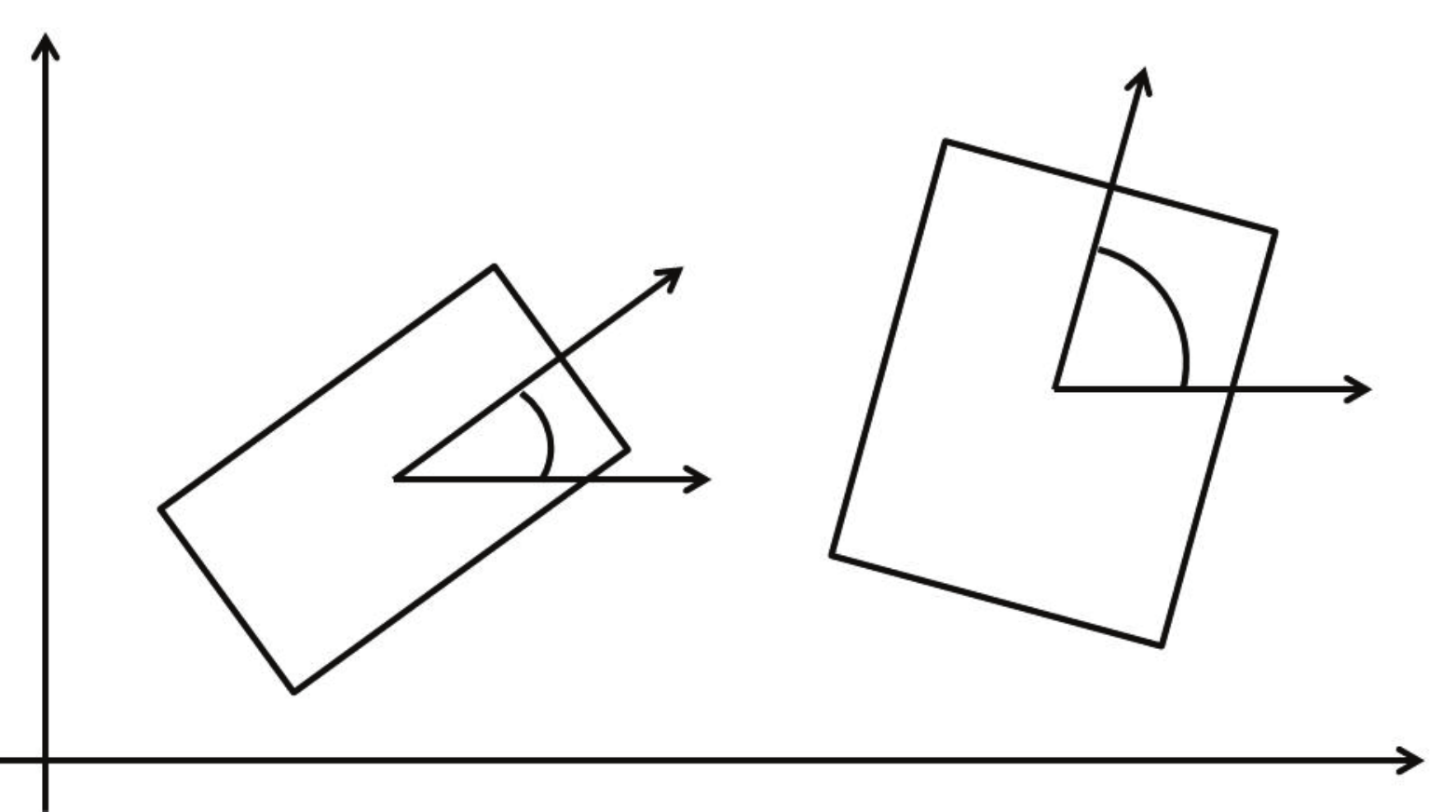}
\setlength{\unitlength}{1cm}
\begin{picture}(0,0)(0,0)
\put(-2,2.3){$\psi_{i}$}
\put(-2.5,1.5){$\mathcal O_i$}
\put(-4.9,1.7){$\psi$}
\put(-5.8,1.1){$\mathcal V$}
  \put(-1,-0.1){$x$}
  \put(-7.4,3.1){$y$}
\end{picture}
\caption{Vehicle and obstacles
\label{fig:veh-obs}
}
\end{center}
\end{figure}

Let $X_i=(x_i,y_i) $ denote the center of obstacle $i$, which may depend on the time $s$.
Let $X=(x,y)$ denote the center of gravity of the vehicle.

\medskip

\paragraph{\bf Circular obstacles.} 
We first consider the simpler case of {circular obstacles} and vehicle, where 
the vehicle is approximated by a closed ball $B(X,r)$ for a given radius $r>0$,
and the obstacles by closed balls $B(X_i(s),r_i)$ centered at $X_i(s)$ and with given fixed radius $r_i>0$.
Then it holds $B(X_i(s),r_i)\cap B(X,r) = \emptyset \equivalent \|X-X_i(s)\|_2 > r+r_i$,
hence there is no collision at time $s$ if 
$B(X,r)\cap \bigcup_{1\leq i\leq k} B(X_i(s),r_i) = \emptyset$, which amounts to saying 
\be \label{eq:gcirc}
   g(z,s):= \min_{1\leq i\leq k}  -(\|X-X_i(s)\|_2-r-r_i) <0,
\ee
where $z=(x,y,\psi,v)$.

\medskip

\paragraph{\bf Rectangular obstacles.} 

We now turn on the more realistic case of {rectangular obstacles} and vehicle (see Fig.~\ref{fig:veh-obs}).

Let us first remark that the Gilbert-Johnson-Keehrti distance algorithm \cite{Gilbert1988} can compute
  the signed distance function for two convex obstacles in two dimensions and therefore can be used to detect collisions.
However, since the obstacle function will have to be computed on each grid point of the discretization method,
in order to be more efficient, we will look for a more straightforward analytic way to detect collision avoidance.

We assume the vehicle $\cV$ is a rectangle centered at $X$ and with half lengths $\ell=(\ell_x,\ell_y)^T$ 
rotated by the angle $\psi$, i.e.,
%
\begin{align*}
   \cV & = X + R_\psi ([-\ell_x,\ell_x] \times [-\ell_y,\ell_y])
\end{align*}
where $R_\psi$ denotes the rotation matrix
$R_\psi:=\MAT{cc}{\cos(\psi) & -\sin(\psi)\\ \sin(\psi) & \cos(\psi)}$.
We assume that each obstacle $\cO_i$ is also a rectangle with center $X^i=X^i(s)$ and
half lengths $\ell^i=(\ell^i_x,\ell^i_y)^T$.
Then the four corners $(X_j)_{1\leq j\leq 4}$ of the vehicle with state $z=(x,y,\psi,v)$ are determined by 
\beno
  X_j = X + R_\psi T_j \ell, \quad 
    T_j=\MAT{cc}{ (-1)^{j-1} & 0 \\ 0 & (-1)^{\lfloor \frac{j-1}{2} \rfloor} }, \quad 1\leq j\leq 4,
\eeno
where $\lfloor x \rfloor$ denotes the integer part of a real $x$,
and in the same way the four corners $(X_j^i)_{1\leq j\leq 4}$
of obstacle $\cO_i$ (determined by its center $X^i(s)$ and orientation $\psi_i(s)$) are given by
\beno
  X^{i}_j := X^i + R_{\psi_i} T_j \ell^i, \quad 1\leq j\leq 4.
\eeno

Furthermore, for a given $X=(x,y)^T$, let
$$
  d_\ell(X):= \max(\ell_x-|x|,\ell_y-|y|).
$$

The function $d_\ell$ is a level set function for the avoidance of $[-\ell_x,\ell_x]\times [-\ell_y,\ell_y]$, since 
$d_\ell(X)<0$ $\equivalent$ $X\notin [-\ell_x,\ell_x]\times [-\ell_y,\ell_y]$.
For an arbitrary point $Y \in \R^2$, the following function
$$ d_\cV(Y):= d_\ell( R_{-\psi}(Y-X) )$$
is a level set function for the avoidance of the vehicle, 
in the sense that $d_\cV(Y)<0$ $\equivalent$ $Y\notin V$. 

In the same way, 
$$ d_{\cO_i}(Y):= d_{\ell^i}( R_{-\psi_i}(Y-X^i) )$$
satisfies $d_{\cO_i}(Y)<0$ $\equivalent$ $Y\notin \cO_i$.

Now we consider the following function, for $z=(x,y,\psi,v)\in \R^4$ and $s \geq 0$:
\be \label{eq:grect}
  g(z,s):= \max_{1\leq i\leq k} \bigg(
    \max_{j=1,\dots,4} d_{\cV(z)}(X^i_j(s)),\ 
    \max_{j=1,\dots,4} d_{\cO_i(s)}(X_j(z)) 
  \bigg).
\ee
If the \MODIF{positions of the obstacles} do not depend of time, we can define $g(z)$ in the same way without time dependency.
Here we have denoted $\cV(z)$ and $X_j(z)$ for the center and the corners of the vehicle to
stress the dependance on the state variable $z$,
and $X^i_j(s)$ to emphasize the time dependency of the obstacle corners.

The function $g$ will serve as a level set function for obstacle avoidance.
Presently, from the definition of the $g$ function, it holds:

\begin{lemma} \label{lem:condition1}
The function $g$ is Lipschitz continuous and 
\be\label{eq:condition1}
  g(z,s)<0  \quad \equivalent \quad 
  \forall i,j:\ X_j(z)\notin \cO_i(s)\ \mbox{and}\  \forall i,j:\ X^i_j(s) \notin \cV(z)
\ee
\end{lemma}

However, we aim to characterize the fact that the obstacle and vehicle are disjoint, i.e.,
\be\label{eq:condition2}
  \cV(z(s))\ \bigcap\ \bigg(\displaystyle \bigcup_{1\leq i\leq k} \cO_i(s)\bigg) = \emptyset.
\ee
In general, the condition $g(z(t_n),t_n)<0$ at a given time $t_n$, i.e., condition \eqref{eq:condition1},
is not sufficient to ensure that \eqref{eq:condition2} holds at time $t_n$,
as shown by the counter-example illustrated in Fig.~\ref{fig:counterex}.

\if{
\begin{figure}
\begin{center}
\includegraphics[angle=0,width=7cm]{\figs/paper2-eps-converted-to.pdf}
\setlength{\unitlength}{1cm}
\begin{picture}(0,0)(0,0)
\put(-3.5,1.1){$\mathcal O_i$}
\put(-5.3,1.6){$\mathcal V$}
  \put(-1,-0.1){$x$}
  \put(-7.4,3.1){$y$}
  \put(-8,0){$(x,y)$}
\end{picture}
\caption{\label{fig:counterex} 
vehicle and obstacle are not necessarily disjoint even if \eqref{eq:condition1} holds}
\end{center}
\end{figure}
}\fi

\begin{figure}
\begin{center}
\begin{picture}(0,0)%
\includegraphics{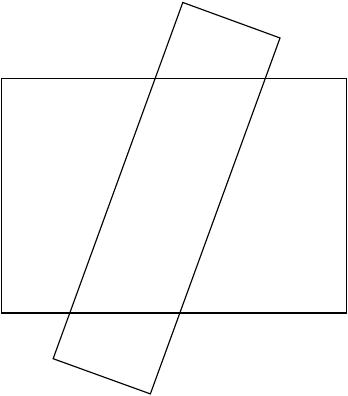}%
\end{picture}%
\setlength{\unitlength}{2901sp}%
\begingroup\makeatletter\ifx\SetFigFont\undefined%
\gdef\SetFigFont#1#2#3#4#5{%
  \reset@font\fontsize{#1}{#2pt}%
  \fontfamily{#3}\fontseries{#4}\fontshape{#5}%
  \selectfont}%
\fi\endgroup%
\begin{picture}(2274,2581)(1488,-1864)
\put(2341,-1636){\makebox(0,0)[lb]{\smash{{\SetFigFont{11}{13.2}{\rmdefault}{\mddefault}{\updefault}{\color[rgb]{0,0,0}$\mathcal O$}%
}}}}
\put(1576,-16){\makebox(0,0)[lb]{\smash{{\SetFigFont{11}{13.2}{\rmdefault}{\mddefault}{\updefault}{\color[rgb]{0,0,0}$\mathcal V$}%
}}}}
\end{picture}%
\caption{\label{fig:counterex} 
Vehicle and obstacle are not necessarily disjoint even if the condition \eqref{eq:condition1} holds, 
i.e., even if each corner point of $\mathcal V$ (resp. $\cO$) lies outside $\cO$ (resp. $\cV$).
}
\end{center}
\end{figure}

\medskip

In order to use condition \eqref{eq:condition1} as a sufficient condition,
we need furthermore the following result.

\begin{lemma}\label{lem:condition2}
We consider that the vehicle as well as all obstacles $\cO_i$ are rectangles.
Let $d$ be the minimal car and obstacle half-lengths 
\be\label{eq:bard}
  \underline{d}:=\min(\ell_x,\ell_y,\min_{1\leq i\leq k}(\min(\ell^i_x,\ell^i_y))) 
\ee
and let $\bar v$ be an upper bound for all relative velocities between the car corners and the obstacle corners 
involved in the computations
\be
\label{eq:barv}
  \overline{v}:= \max_{s\in[0,T]} \max_{1\leq j,j'\leq 4}\max_{1\leq i\leq k}
     \|\dot X_j(s)-\dot X^k_{j'}(s)\|.
\ee

Assume that 
\begin{itemize}
\item[$(i)$]  at time $s=t_n$, the vehicle is disjoint from the obstacles (i.e., \eqref{eq:condition2} holds);
\item[$(ii)$] at time $s=t_{n+1}$ condition \eqref{eq:condition1} is fulfilled (i.e., \mbox{$g(z(t_{n+1}),t_{n+1})<0)$};
\item[$(iii)$] the time step satisfies 
\be\label{eq:dt-lowerbound}
  \Dt_n:=t_{n+1}-t_n\ <\ \underline{d}/\overline{v}.
\ee
\end{itemize}
Then at time $s=t_{n+1}$ the vehicle is also disjoint from the obstacles (i.e., \eqref{eq:condition2} holds with $s=t_{n+1}$).
\end{lemma}

\begin{remark}
An upper bound of \eqref{eq:barv}, for $\overline{v}$,
can be obtained on a given computational bounded domain $\Omega$ and for given obstacle parameters.
\end{remark}

\proof{of Lemma~\ref{lem:condition2}.}
{
Assume, to the contrary, that an intersection occurs at time $t_{n+1}$, 
i.e., \eqref{eq:condition2} is not satisfied for a given obstacle.
while \eqref{eq:condition1} is fulfilled at time $t_{n+1}$ due to assumption~(ii). 
Let us show, in that case, that one point $M(s)$ of the obstacle (or, resp., the vehicle) will run through the vehicle
(or, resp., the obstacle) for a distance of at least $d\,\geq\, \underline{d}$.
This will imply in particular that 
\be\label{eq:dt-lowerbound2}
  \overline{v}\, \Dt_n\, \geq \, \int_{t_n}^{t_{n+1}} v(s) ds \, {\geq} \, \underline{d}  
\ee
(where $v(s)$ denotes the relative velocity of the considered point $M(s)$),
which will contradict \eqref{eq:dt-lowerbound}.

Precisely, assume that one obstacle (say $\cO:=\cO^i$) intersects the vehicle $\cV$ at time $t_{n+1}$.
Let $\ell_{\cV}:=\min(\ell_x,\ell_y)$ denote the minimum half-length of the edges of the vehicle, as well as 
$\ell_{\cO}:=\min(\ell_x^i,\ell_y^i)$.

Because $\cO$ and $\cV$ will play a symmetric role from now on, we can also assume that $\ell_\cO \geq \ell_\cV$ (otherwise we may exchange $\cO$ and $\cV$ 
in the forthcoming argument).

Let $I_{\cV}$ denote the set of points $x \in \cV$ such that $d(x,\R^2 {\setminus} \cV)$ 
is maximal. We notice that $I_{\cV}$ is a segment 
and that 
$$
  I_{\cV}:=\{x\in \R^2 \,\vert\, d(x,\R^2 {\setminus} \cV)=\ell_{\cV}\}.
$$
($I_\cV$ may be reduced to a point for a squared vehicle.)
This set is illustrated in Figure~\ref{fig:setIV}.

\begin{figure}
\begin{center}
\begin{picture}(0,0)%
\includegraphics{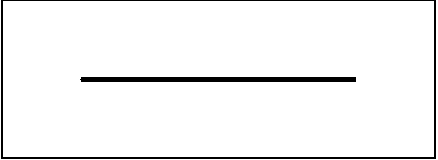}%
\end{picture}%
\setlength{\unitlength}{1657sp}%
\begingroup\makeatletter\ifx\SetFigFont\undefined%
\gdef\SetFigFont#1#2#3#4#5{%
  \reset@font\fontsize{#1}{#2pt}%
  \fontfamily{#3}\fontseries{#4}\fontshape{#5}%
  \selectfont}%
\fi\endgroup%
\begin{picture}(4994,1844)(879,-1433)
\put(1126,-1186){\makebox(0,0)[lb]{\smash{{\SetFigFont{6}{7.2}{\rmdefault}{\mddefault}{\updefault}{\color[rgb]{0,0,0}$\cV$}%
}}}}
\put(1801,-286){\makebox(0,0)[lb]{\smash{{\SetFigFont{6}{7.2}{\rmdefault}{\mddefault}{\updefault}{\color[rgb]{0,0,0}$I_{\cV}$}%
}}}}
\end{picture}%
\end{center}
\caption{\label{fig:setIV} The interval $I_\cV$ associated to the vehicle $\cV$.}
\end{figure}


In the same way for the given obstacle $\cO=\cO^i$, let
$$ I_{\cO}:=\{x\in \R^2 {\,\vert\,} d(x,\R^2 {\setminus} \cO)=\ell_{\cO}\}.
   \rule[-1ex]{0ex}{1ex} 
$$

Let us first show that (at time $t_{n+1}$): 
$$  
  \cO\cap I_\cV = \emptyset.
$$ 
If, to the contrary, $\cO\cap I_\cV\neq\emptyset$ (at time $t_{n+1}$), then 
the point $M(t_{n+1})\in \cO\cap I_\cV$, attached to $\cO$,
has run through the 
vehicle for a distance of at least $\ell_\cV$ between $t_n$ and $t_{n+1}$ (since at time $t_n$, by assumption $(i)$, 
$M(t_n)\in \cO(t_n)$ was outside of $\cV(z(t_n))$).
By using the fact that $\ell_\cV\geq \underline{d}$, and \eqref{eq:dt-lowerbound2}, we 
arrive to the contradiction
$\overline{v}\Dt \geq \underline{d}$.

In the same way, we can conclude to
$$
  \cV\cap I_\cO = \emptyset.
$$
%


Let us now consider an edge of the obstacle (say an edge $[A;B]$ of $\cO$) that overlaps the vehicle $\cV$, but with $A,B \notin \cV$,
as illustrated in Figs.~\ref{fig:setIV2-1}-\ref{fig:setIV2-23}.
Furthermore we notice that $[A;B]\cap I_\cV\subset \cO\cap I_\cV=\emptyset$, hence $[A;B]\cap I_\cV=\emptyset$.

Notice that for any edge of $\cO$ that overlaps the vehicle,
and since the corner points are outside of $\cV$, 
there can be only two generic situations: 
either the edge intersects two parallel edges of the vehicle, or the edge intersects two neighboring edges of the vehicle.
(The particular cases when one edge overlaps the vehicle and intersects another edge, or overlaps one corner point cannot happen,
since each corner point of each considered object is assumed to be strictly outside the other object.)

Also, it is not possible that $I_\cV$ lies in between two parallel edges,
say between the edges $(A,B)$ and $(A',B')$ as in Fig.~\ref{fig:setIV2-imp-0}, since otherwise the segment $I_\cV$ would be fully 
included in the object $\cO=(ABB'A')$ and this case has already been excluded. 
Some forbidden situations are also depicted in~Fig.~\ref{fig:setIV2-imp-0}--\ref{fig:setIV2-imp-2}. 

\begin{figure}[!hbtp]
\begin{center}

\vspace{-0.5cm}
\end{center}
\caption{\label{fig:setIV2-imp-0} Forbidden situation: $\cO\cap I_\cV\neq\emptyset$}
\end{figure}

\begin{figure}[!hbtp]
\begin{center}
\begin{tabular}{cc}
\begin{picture}(0,0)%
\includegraphics{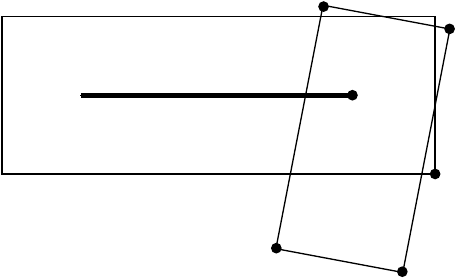}%
\end{picture}%
\setlength{\unitlength}{1657sp}%
\begingroup\makeatletter\ifx\SetFigFont\undefined%
\gdef\SetFigFont#1#2#3#4#5{%
  \reset@font\fontsize{#1}{#2pt}%
  \fontfamily{#3}\fontseries{#4}\fontshape{#5}%
  \selectfont}%
\fi\endgroup%
\begin{picture}(5351,3413)(879,-2596)
\put(4216,-2191){\makebox(0,0)[lb]{\smash{{\SetFigFont{6}{7.2}{\rmdefault}{\mddefault}{\updefault}{\color[rgb]{0,0,0}A}%
}}}}
\put(1801,-286){\makebox(0,0)[lb]{\smash{{\SetFigFont{6}{7.2}{\rmdefault}{\mddefault}{\updefault}{\color[rgb]{0,0,0}$I_{\cV}$}%
}}}}
\put(1036,-1276){\makebox(0,0)[lb]{\smash{{\SetFigFont{6}{7.2}{\rmdefault}{\mddefault}{\updefault}{\color[rgb]{0,0,0}$\cV$}%
}}}}
\put(4771,-871){\makebox(0,0)[lb]{\smash{{\SetFigFont{6}{7.2}{\rmdefault}{\mddefault}{\updefault}{\color[rgb]{0,0,0}P}%
}}}}
\put(6031,-1501){\makebox(0,0)[lb]{\smash{{\SetFigFont{6}{7.2}{\rmdefault}{\mddefault}{\updefault}{\color[rgb]{0,0,0}C}%
}}}}
\put(4621,607){\makebox(0,0)[lb]{\smash{{\SetFigFont{6}{7.2}{\rmdefault}{\mddefault}{\updefault}{\color[rgb]{0,0,0}B}%
}}}}
\put(6136,322){\makebox(0,0)[lb]{\smash{{\SetFigFont{6}{7.2}{\rmdefault}{\mddefault}{\updefault}{\color[rgb]{0,0,0}$B^\prime$}%
}}}}
\put(5626,-2498){\makebox(0,0)[lb]{\smash{{\SetFigFont{6}{7.2}{\rmdefault}{\mddefault}{\updefault}{\color[rgb]{0,0,0}$A^\prime$}%
}}}}
\end{picture}%
\quad
& 
\quad
\begin{picture}(0,0)%
\includegraphics{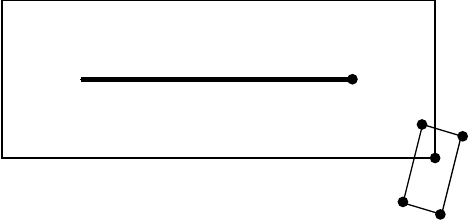}%
\end{picture}%
\setlength{\unitlength}{1657sp}%
\begingroup\makeatletter\ifx\SetFigFont\undefined%
\gdef\SetFigFont#1#2#3#4#5{%
  \reset@font\fontsize{#1}{#2pt}%
  \fontfamily{#3}\fontseries{#4}\fontshape{#5}%
  \selectfont}%
\fi\endgroup%
\begin{picture}(5429,2833)(879,-2422)
\put(6060,-2326){\makebox(0,0)[lb]{\smash{{\SetFigFont{6}{7.2}{\rmdefault}{\mddefault}{\updefault}{\color[rgb]{0,0,0}$A^\prime$}%
}}}}
\put(1801,-286){\makebox(0,0)[lb]{\smash{{\SetFigFont{6}{7.2}{\rmdefault}{\mddefault}{\updefault}{\color[rgb]{0,0,0}$I_{\cV}$}%
}}}}
\put(1036,-1276){\makebox(0,0)[lb]{\smash{{\SetFigFont{6}{7.2}{\rmdefault}{\mddefault}{\updefault}{\color[rgb]{0,0,0}$\cV$}%
}}}}
\put(4771,-871){\makebox(0,0)[lb]{\smash{{\SetFigFont{6}{7.2}{\rmdefault}{\mddefault}{\updefault}{\color[rgb]{0,0,0}P}%
}}}}
\put(5716,-1711){\makebox(0,0)[lb]{\smash{{\SetFigFont{6}{7.2}{\rmdefault}{\mddefault}{\updefault}{\color[rgb]{0,0,0}C}%
}}}}
\put(5498,-909){\makebox(0,0)[lb]{\smash{{\SetFigFont{6}{7.2}{\rmdefault}{\mddefault}{\updefault}{\color[rgb]{0,0,0}B}%
}}}}
\put(6293,-1065){\makebox(0,0)[lb]{\smash{{\SetFigFont{6}{7.2}{\rmdefault}{\mddefault}{\updefault}{\color[rgb]{0,0,0}$B^\prime$}%
}}}}
\put(5131,-2064){\makebox(0,0)[lb]{\smash{{\SetFigFont{6}{7.2}{\rmdefault}{\mddefault}{\updefault}{\color[rgb]{0,0,0}A}%
}}}}
\end{picture}%
\end{tabular}
\vspace{-0.5cm}
\end{center}
\caption{\label{fig:setIV2-imp-1} Forbidden situations: $\cO\cap I_\cV\neq\emptyset$ (left), 
        in the right figure, \eqref{eq:condition1} does not hold since $B \in \cV$ and $C \in \cO$}
\end{figure}

\begin{figure}[!hbtp]
\begin{center}
\begin{tabular}{cc}
\begin{picture}(0,0)%
\includegraphics{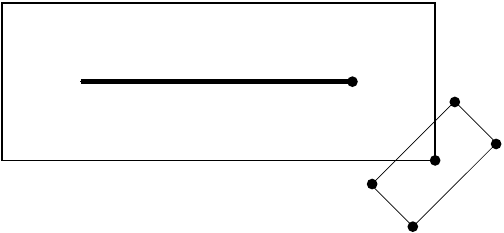}%
\end{picture}%
\setlength{\unitlength}{1657sp}%
\begingroup\makeatletter\ifx\SetFigFont\undefined%
\gdef\SetFigFont#1#2#3#4#5{%
  \reset@font\fontsize{#1}{#2pt}%
  \fontfamily{#3}\fontseries{#4}\fontshape{#5}%
  \selectfont}%
\fi\endgroup%
\begin{picture}(5819,3044)(879,-2633)
\put(6683,-1357){\makebox(0,0)[lb]{\smash{{\SetFigFont{6}{7.2}{\rmdefault}{\mddefault}{\updefault}{\color[rgb]{0,0,0}$B^\prime$}%
}}}}
\put(1801,-286){\makebox(0,0)[lb]{\smash{{\SetFigFont{6}{7.2}{\rmdefault}{\mddefault}{\updefault}{\color[rgb]{0,0,0}$I_{\cV}$}%
}}}}
\put(1036,-1276){\makebox(0,0)[lb]{\smash{{\SetFigFont{6}{7.2}{\rmdefault}{\mddefault}{\updefault}{\color[rgb]{0,0,0}$\cV$}%
}}}}
\put(5716,-1711){\makebox(0,0)[lb]{\smash{{\SetFigFont{6}{7.2}{\rmdefault}{\mddefault}{\updefault}{\color[rgb]{0,0,0}C}%
}}}}
\put(4762,-854){\makebox(0,0)[lb]{\smash{{\SetFigFont{6}{7.2}{\rmdefault}{\mddefault}{\updefault}{\color[rgb]{0,0,0}P}%
}}}}
\put(4982,-2011){\makebox(0,0)[lb]{\smash{{\SetFigFont{6}{7.2}{\rmdefault}{\mddefault}{\updefault}{\color[rgb]{0,0,0}A}%
}}}}
\put(5595,-2537){\makebox(0,0)[lb]{\smash{{\SetFigFont{6}{7.2}{\rmdefault}{\mddefault}{\updefault}{\color[rgb]{0,0,0}$A^\prime$}%
}}}}
\put(6158,-692){\makebox(0,0)[lb]{\smash{{\SetFigFont{6}{7.2}{\rmdefault}{\mddefault}{\updefault}{\color[rgb]{0,0,0}B}%
}}}}
\end{picture}%
\quad
& 
\quad
\begin{picture}(0,0)%
\includegraphics{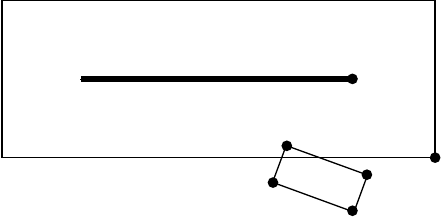}%
\end{picture}%
\setlength{\unitlength}{1657sp}%
\begingroup\makeatletter\ifx\SetFigFont\undefined%
\gdef\SetFigFont#1#2#3#4#5{%
  \reset@font\fontsize{#1}{#2pt}%
  \fontfamily{#3}\fontseries{#4}\fontshape{#5}%
  \selectfont}%
\fi\endgroup%
\begin{picture}(5313,2840)(879,-2429)
\put(5175,-1815){\makebox(0,0)[lb]{\smash{{\SetFigFont{6}{7.2}{\rmdefault}{\mddefault}{\updefault}{\color[rgb]{0,0,0}$B^\prime$}%
}}}}
\put(1801,-286){\makebox(0,0)[lb]{\smash{{\SetFigFont{6}{7.2}{\rmdefault}{\mddefault}{\updefault}{\color[rgb]{0,0,0}$I_{\cV}$}%
}}}}
\put(1036,-1276){\makebox(0,0)[lb]{\smash{{\SetFigFont{6}{7.2}{\rmdefault}{\mddefault}{\updefault}{\color[rgb]{0,0,0}$\cV$}%
}}}}
\put(4771,-871){\makebox(0,0)[lb]{\smash{{\SetFigFont{6}{7.2}{\rmdefault}{\mddefault}{\updefault}{\color[rgb]{0,0,0}P}%
}}}}
\put(5993,-1501){\makebox(0,0)[lb]{\smash{{\SetFigFont{6}{7.2}{\rmdefault}{\mddefault}{\updefault}{\color[rgb]{0,0,0}C}%
}}}}
\put(4928,-2333){\makebox(0,0)[lb]{\smash{{\SetFigFont{6}{7.2}{\rmdefault}{\mddefault}{\updefault}{\color[rgb]{0,0,0}$A^\prime$}%
}}}}
\put(4073,-1111){\makebox(0,0)[lb]{\smash{{\SetFigFont{6}{7.2}{\rmdefault}{\mddefault}{\updefault}{\color[rgb]{0,0,0}B}%
}}}}
\put(3705,-1704){\makebox(0,0)[lb]{\smash{{\SetFigFont{6}{7.2}{\rmdefault}{\mddefault}{\updefault}{\color[rgb]{0,0,0}A}%
}}}}
\end{picture}%
\end{tabular}
\vspace{-0.5cm}
\end{center}
\caption{\label{fig:setIV2-imp-2} Forbidden situations: \eqref{eq:condition1} does not hold, since $C \in \cO$ (left), or $B \in \cV$ (right)}
\end{figure}


We are left with three possible situations, depicted in Figures~\ref{fig:setIV2-1} and \ref{fig:setIV2-23} (left and right):
\begin{itemize}
\item{case 1:}
  each one of two parallel edges $[A;B]$ and $[A'; B']$ of $\cO$ intersect two parallel edges of $\cV$, as in Fig.~\ref{fig:setIV2-1};
\item{case 2:}
  only one edge $[A;B]$ of $\cO$ intersects two parallel edges of $\cV$, as illustrated in Fig.~\ref{fig:setIV2-23} (left);
\item{case 3:}
  no edge of $\cO$ intersect two parallel edges of $\cV$; rather, the two edges $[A;B]$ and $[A';B']$ intersects 
  two neighboring edges of $\cV$, as in Fig.~\ref{fig:setIV2-23} (right).
\end{itemize}


In all these cases, the distance between the parallel edges $[A;B]$ and $[A',B']$ of $\cO$ (which is greater or equal to $2\ell_\cO$),
is smaller than the distance from $P$ (an extremal point of $I_\cV$) to a corner point $C$ of $\cV$ (see Figures~\ref{fig:setIV2-1}-\ref{fig:setIV2-23}),
with $d(P,C)= d(C,I_\cV)=\sqrt{2}\ell_\cV$.
Therefore 
$$ 2 \ell_\cO \leq d(P,C) = \sqrt{2}\ell_\cV. $$
This is in contradiction with the fact that $\ell_\cV \leq \ell_\cO$ (and $\ell_\cV>0$), and concludes the proof of the Lemma.
\endproof

\begin{figure}[!hbtp]
\begin{center}
\begin{tabular}{cc}
\begin{picture}(0,0)%
\includegraphics{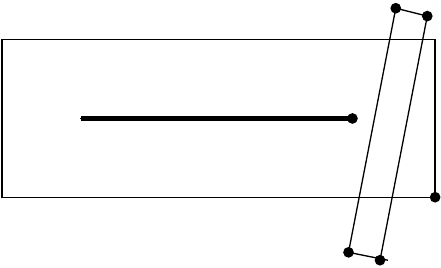}%
\end{picture}%
\setlength{\unitlength}{1657sp}%
\begingroup\makeatletter\ifx\SetFigFont\undefined%
\gdef\SetFigFont#1#2#3#4#5{%
  \reset@font\fontsize{#1}{#2pt}%
  \fontfamily{#3}\fontseries{#4}\fontshape{#5}%
  \selectfont}%
\fi\endgroup%
\begin{picture}(5351,3456)(879,-2452)
\put(4501,-2086){\makebox(0,0)[lb]{\smash{{\SetFigFont{6}{7.2}{\rmdefault}{\mddefault}{\updefault}{\color[rgb]{0,0,0}A}%
}}}}
\put(1801,-286){\makebox(0,0)[lb]{\smash{{\SetFigFont{6}{7.2}{\rmdefault}{\mddefault}{\updefault}{\color[rgb]{0,0,0}$I_{\cV}$}%
}}}}
\put(1036,-1276){\makebox(0,0)[lb]{\smash{{\SetFigFont{6}{7.2}{\rmdefault}{\mddefault}{\updefault}{\color[rgb]{0,0,0}$\cV$}%
}}}}
\put(5986,569){\makebox(0,0)[lb]{\smash{{\SetFigFont{6}{7.2}{\rmdefault}{\mddefault}{\updefault}{\color[rgb]{0,0,0}$B^\prime$}%
}}}}
\put(4771,-871){\makebox(0,0)[lb]{\smash{{\SetFigFont{6}{7.2}{\rmdefault}{\mddefault}{\updefault}{\color[rgb]{0,0,0}P}%
}}}}
\put(5131,794){\makebox(0,0)[lb]{\smash{{\SetFigFont{6}{7.2}{\rmdefault}{\mddefault}{\updefault}{\color[rgb]{0,0,0}B}%
}}}}
\put(5356,-2356){\makebox(0,0)[lb]{\smash{{\SetFigFont{6}{7.2}{\rmdefault}{\mddefault}{\updefault}{\color[rgb]{0,0,0}$A^\prime$}%
}}}}
\put(6031,-1501){\makebox(0,0)[lb]{\smash{{\SetFigFont{6}{7.2}{\rmdefault}{\mddefault}{\updefault}{\color[rgb]{0,0,0}C}%
}}}}
\end{picture}%
\quad
& 
\quad
\end{tabular}
\vspace{-0.5cm}
\end{center}
\caption{\label{fig:setIV2-1} Case 1: two parallel edges of $\cO$ intersect two parallel edges of $\cV$.
}
\end{figure}

\begin{figure}[!hbtp]
\begin{center}
\begin{tabular}{lr}
\begin{picture}(0,0)%
\includegraphics{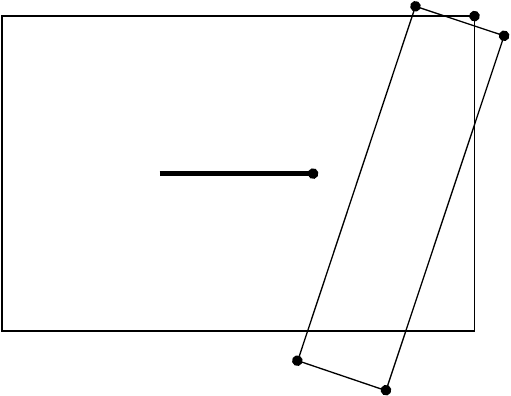}%
\end{picture}%
\setlength{\unitlength}{1657sp}%
\begingroup\makeatletter\ifx\SetFigFont\undefined%
\gdef\SetFigFont#1#2#3#4#5{%
  \reset@font\fontsize{#1}{#2pt}%
  \fontfamily{#3}\fontseries{#4}\fontshape{#5}%
  \selectfont}%
\fi\endgroup%
\begin{picture}(5977,5256)(879,-4387)
\put(6841,209){\makebox(0,0)[lb]{\smash{{\SetFigFont{6}{7.2}{\rmdefault}{\mddefault}{\updefault}{\color[rgb]{0,0,0}$B^\prime$}%
}}}}
\put(1036,-2266){\makebox(0,0)[lb]{\smash{{\SetFigFont{6}{7.2}{\rmdefault}{\mddefault}{\updefault}{\color[rgb]{0,0,0}$\cV$}%
}}}}
\put(2521,-1231){\makebox(0,0)[lb]{\smash{{\SetFigFont{6}{7.2}{\rmdefault}{\mddefault}{\updefault}{\color[rgb]{0,0,0}$I_{\cV}$}%
}}}}
\put(4321,-1231){\makebox(0,0)[lb]{\smash{{\SetFigFont{6}{7.2}{\rmdefault}{\mddefault}{\updefault}{\color[rgb]{0,0,0}$P$}%
}}}}
\put(5356,-4291){\makebox(0,0)[lb]{\smash{{\SetFigFont{6}{7.2}{\rmdefault}{\mddefault}{\updefault}{\color[rgb]{0,0,0}$A^\prime$}%
}}}}
\put(6436,-3301){\makebox(0,0)[lb]{\smash{{\SetFigFont{6}{7.2}{\rmdefault}{\mddefault}{\updefault}{\color[rgb]{0,0,0}$C$}%
}}}}
\put(3871,-3841){\makebox(0,0)[lb]{\smash{{\SetFigFont{6}{7.2}{\rmdefault}{\mddefault}{\updefault}{\color[rgb]{0,0,0}A}%
}}}}
\put(5446,659){\makebox(0,0)[lb]{\smash{{\SetFigFont{6}{7.2}{\rmdefault}{\mddefault}{\updefault}{\color[rgb]{0,0,0}B}%
}}}}
\end{picture}%

\quad &  \quad
\begin{picture}(0,0)%
\includegraphics{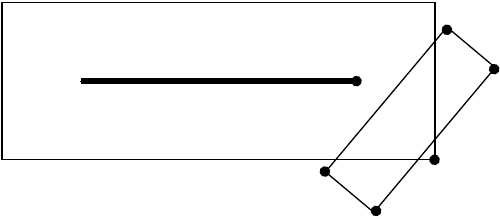}%
\end{picture}%
\setlength{\unitlength}{1657sp}%
\begingroup\makeatletter\ifx\SetFigFont\undefined%
\gdef\SetFigFont#1#2#3#4#5{%
  \reset@font\fontsize{#1}{#2pt}%
  \fontfamily{#3}\fontseries{#4}\fontshape{#5}%
  \selectfont}%
\fi\endgroup%
\begin{picture}(5842,2854)(879,-2443)
\put(6121,164){\makebox(0,0)[lb]{\smash{{\SetFigFont{6}{7.2}{\rmdefault}{\mddefault}{\updefault}{\color[rgb]{0,0,0}B}%
}}}}
\put(1801,-286){\makebox(0,0)[lb]{\smash{{\SetFigFont{6}{7.2}{\rmdefault}{\mddefault}{\updefault}{\color[rgb]{0,0,0}$I_{\cV}$}%
}}}}
\put(1036,-1276){\makebox(0,0)[lb]{\smash{{\SetFigFont{6}{7.2}{\rmdefault}{\mddefault}{\updefault}{\color[rgb]{0,0,0}$\cV$}%
}}}}
\put(5053,-392){\makebox(0,0)[lb]{\smash{{\SetFigFont{6}{7.2}{\rmdefault}{\mddefault}{\updefault}{\color[rgb]{0,0,0}P}%
}}}}
\put(5932,-1540){\makebox(0,0)[lb]{\smash{{\SetFigFont{6}{7.2}{\rmdefault}{\mddefault}{\updefault}{\color[rgb]{0,0,0}C}%
}}}}
\put(6706,-331){\makebox(0,0)[lb]{\smash{{\SetFigFont{6}{7.2}{\rmdefault}{\mddefault}{\updefault}{\color[rgb]{0,0,0}$B^\prime$}%
}}}}
\put(5221,-2356){\makebox(0,0)[lb]{\smash{{\SetFigFont{6}{7.2}{\rmdefault}{\mddefault}{\updefault}{\color[rgb]{0,0,0}$A^\prime$}%
}}}}
\put(4501,-1951){\makebox(0,0)[lb]{\smash{{\SetFigFont{6}{7.2}{\rmdefault}{\mddefault}{\updefault}{\color[rgb]{0,0,0}A}%
}}}}
\end{picture}%
\end{tabular}
\vspace{-0.5cm}
\end{center}
\caption{\label{fig:setIV2-23} 
Case 2 (left): only one edge of $\cO$ (the segment $[A;B]$) intersects two parallel edges of $\cV$.
Case 3 (right): no edge of $\cO$ intersects two parallel edges of $\cV$.
}
\end{figure}




From Lemma~\ref{lem:condition1} and Lemma~\ref{lem:condition2} we deduce, in other words, 
assuming that $\Dt_n\leq \overline{\Dt}:=\underline{d}/\overline{v}$,
that if there is no collision at time $t_n$, then \eqref{eq:condition1}
characterizes the fact that there is no collision at time $t_{n+1}$.

\begin{corol}
In particular, considering a trajectory reconstruction with time steps $t_n=n\Dt$, 
if at the initial time $t_0=0$ the vehicle and the obstacles are disjoint
(i.e, condition \eqref{eq:condition2} holds)
and if $0\leq \Dt<{\overline \Dt}$ where ${\overline\Dt:=\bar d/\bar v}$ and $\bar d$ and $\bar v$ are as in 
\eqref{eq:bard}-\eqref{eq:barv},
and if 
\be\label{eq:gz_negative}
  g(z(t_n),t_n)<0,\quad  \forall n=1,\dots,N
\ee
(i.e., condition \eqref{eq:condition1} holds at each successive time step) 
then there is no collision 
at all discrete time steps $t_n$, i.e.:
\beno
  {\mathcal V}(z(t_n))\ \bigcap\ \bigg( \displaystyle \bigcup_{1\leq i\leq k} \cO_i(t_n)\bigg) = \emptyset\qquad  
    \forall n=1,\dots,N.
\eeno
\end{corol}

\begin{remark}
By using Lemma~\ref{lem:3.4}, the condition \eqref{eq:gz_negative} will be fulfilled along all time steps of the
trajectory reconstruction.

Therefore the obstacle function $g$ can be used for collision avoidance in our HJB framework as long as the 
time-step condition 
\eqref{eq:dt-lowerbound} holds.
\end{remark}

\medskip

\paragraph{\bf Moving obstacles.} 
There are many possible motions for the obstacles, which can be considered. One of them is
a linear motion along a straight path:
\begin{subequations}\label{eq:obs_motion_line}
\be
  x_i(s) & = & x_{i0}+v_{ix}\, s, \\
  y_i(s) & = & y_{i0}+v_{iy}\, s, \\
  \psi_{i0} & = & \arctan\bigg(\frac{v_{iy}}{v_{ix}}\bigg),
\ee
\end{subequations}
where $(v_{ix},v_{iy})$ is the constant velocity of the obstacle $i$
and $\mt_{i0}$ is the initial angle which is constant during the motion.
Another possibility is a motion along a curved road, e.g.~a rotation around a center $(c_x,c_y)$
with constant angular velocity $\omega_i$ starting with the angle $\mt_{i0}$, i.e.
\begin{subequations}\label{eq:obs_motion_curv}
\be
 x_{i}(s) & = & a\cos(\mt_{i0}+w_i s)+c_x, \\
 y_{i}(s) & = & a\sin(\mt_{i0}+w_i s)+c_y, \\
 \psi_{i0}(s) & = & (\mt_{i0}+w_i s)-\frac{\pi}{2}.
\ee
\end{subequations}

\medskip\noindent

\paragraph{\bf Collision avoidance between time steps.}

It is still possible that a violation of collision avoidance appears between time steps $t_n$ and $t_{n+1}$.
In order to control this error or to avoid it, we first start with the following result.

Let $d(A,B)$ be the signed distance between two sets $A,B$, defined by 
$$ d(A,B):=\min(\mu(A,B),\mu(B,A)) $$
where $\mu(A,B):=\min_{x\in A} d_B(x)$ and $d_B(.)$ is the signed distance to $B$. 
We notice that if $A\cap B=\emptyset \Leftrightarrow d(A,B)>0$, and in that case 
$d(A,B)=\min_{x\in A, y\in B} \|x-y\|_2$.

Let $d(t): = d(\cV(z(t)), \displaystyle\cup_i \cO_i(t))$ be the signed distance between the vehicle and the set of obstacles.

\begin{lemma}\label{lem:avoidance2}
Let $\bar v$ be the maximum relative velocity between the vehicle and the obstacles (as in \eqref{eq:barv}), 
and $\Dt_n:=t_{n+1}-t_n$.
\\
$(i)$ It holds
\be\label{eq:dhineq}
  d(t) \geq  \min\big(d(t_n),d(t_{n+1})\big) -{\bar v}\frac{\Dt_n}{2},\quad \forall t\in [t_n,t_{n+1}].
\ee
$(ii)$
In the same way, in the case of the vehicle and the obstacles are modelized as rectangles, using \eqref{eq:grect},
it holds 
\be \label{eq:g-ineq}
  g(z(t),t)\leq \max(g(z(t_n),t_n),g(z(t_{n+1}),t_{n+1}))+\bar v \frac{\Dt_n}{2} \quad  \forall t\in[t_n,t_{n+1}].
\ee
\end{lemma}
\begin{proof}
We only give a sketch of the proof of $(i)$.
The maximum distance that a point of the vehicle can cover during a time step $\tau$ is $\tau \bar v$. 
Taking into account the relative velocity, we obtain that on the time interval $t\in [t_n,t_n+\Dt_n/2]$,
$d(t)\geq d(t_n)-{\bar v}(\Dt_n/2)$.
In the same way one can prove that for $t\in [t_n+\Dt/2,t_{n+1}]$, $d(t)\geq d(t_{n+1})- {\bar v} (\Dt_n/2)$.
Hence the desired formula \eqref{eq:dhineq} follows.
\end{proof}

{\em Minimal distance between time steps.}
As a first result, if 
$$\min(d(t_n),d(t_{n+1}))\geq 0$$
(no collision at time steps $t_n$ and $t_{n+1}$), then
$$
  \min\limits_{t \in [t_n,t_{n+1}]} d(t)\geq -\bar v \frac{\Dt_n}{2}.
$$

{\em Secure collision avoidance.} 
On the other hand if we want to secure collision avoidance, 
we can consider a small number $\eps>0$ 
and require that
\be\label{eq:spec_enlarged}
  \min(d(t_n),d(t_{n+1}))>\eps,
\ee
Then using \eqref{eq:dhineq} it holds:
\be\label{eq:dtgeq0}
  \frac{\Dt_n}{2} < \eps/\bar v \quad \Rightarrow \quad
  \forall t\in [t_n,t_{n+1}],\quad d(t)>0.
\ee




\begin{remark}
Consider the case of vehicle and obstacle being rectangles, let $g$ be defined as in~\eqref{eq:grect},
$\tilde g_\eps$ the following obstacle function
$$
  \tilde g_\eps := g+\eps,
$$
and assume that 
\be \label{eq:epscollavoid}
  \max(\tilde g_\eps(z(t_n),t_n),\tilde g_\eps(z(t_{n+1}),t_{n+1}))\leq 0,
\ee
then using \eqref{eq:g-ineq}, it holds
\beno
  \frac{\Dt_n}{2}\leq \eps/{\bar v} \quad \Rightarrow \quad
  \max_{t \in[t_n,t_{n+1}]} g(z(t),t) \leq -\eps + \bar v \frac{\Dt_n}{2} \leq 0.
\eeno
Therefore, if there is collision avoidance neither at time $t_n$ nor at time $t_{n+1}$ within  a margin of length $\eps>0$,
(i.e., if \eqref{eq:epscollavoid} holds),
and if $\Dt_n/2\leq \eps/{\bar v}$,
then collision avoidance will also hold on $[t_n,t_{n+1}]$.
\end{remark}

\section{Numerical Simulations}
Here we plot some numerical simulations for different scenarios. 
Throughout the paper,
we use the 4-dimensional point mass model~\eqref{eq:model}.
In the present work and for solving the HJ equation \eqref{hjb-t},
we have used an Essentially Non Oscillating (ENO) finite difference scheme of second order for the spatial discretization of the
Hamilton-Jacobi-Bellman partial differential equation.
It is coupled with an Euler forward scheme (RK1) in time
(see \cite{osh-shu-91}).
The computations have been performed by using the ROC-HJ parallel solver~\cite{rochj}.

\subsection{Scenario~1: straight road with a fixed rectangular obstacle.}

We first consider a straight road configuration with only one fixed obstacle,
as in Figure~\ref{scenario_case}. The aim for the vehicle (represented by the red box)
is to minimize the time to reach the target area (the target, delimited by a blue line)
and to avoid the obstacle (blue box).
More precisely the parameters of the problem are:

\begin{itemize}
   \item Vehicle parameters and initial position:\\
         half lengths $\ell_x=\ell_y=1.0$,  \\
         initial value $z_0=(x_0,y_0,\mt_0,v_0)=(-40.0,\ -1.5,\ 0.0,\ 35.0)$
   \item Target: $\Omega=\{(x,y),\ x\geq 0,\ |y|\leq 3.5\}$. 
   \item Obstacle parameters:
         one fixed obstacle with half lengths $\ell^1_x=1.0$ and $\ell^1_y=1.0$, centered at 
         $X^1=(-10.0,\ -1.5)$
   \item Road parameters:
         straight road as described in \eqref{eq:Kr_straight}.
\end{itemize}

\begin{figure}
\includegraphics[width=10cm,angle=0]{\figs/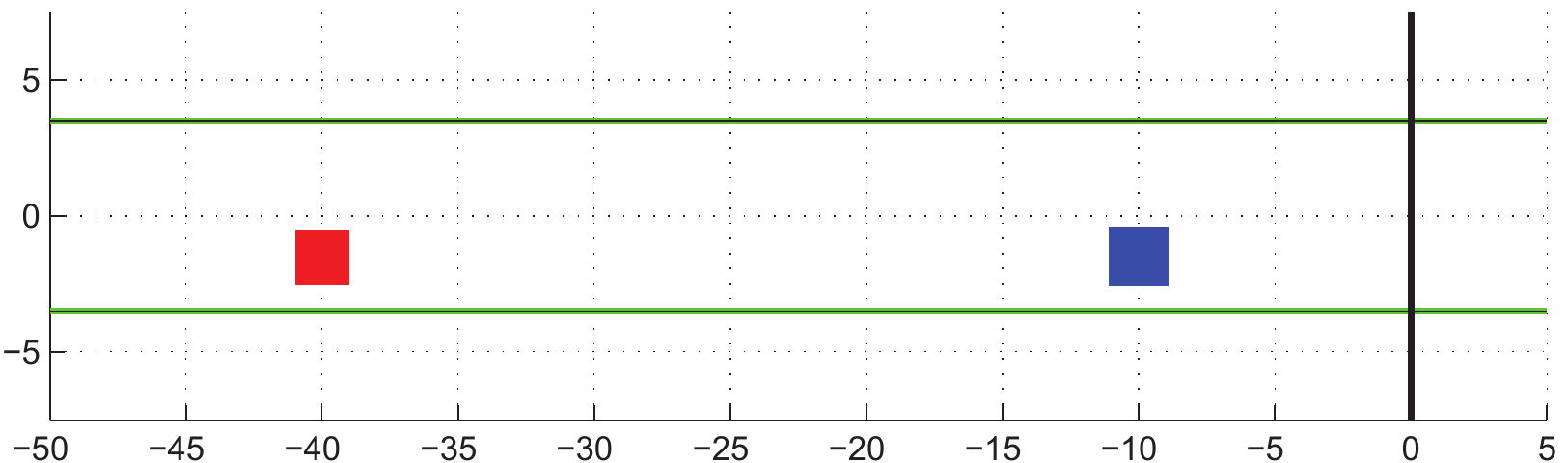} 
  \caption{
  \label{scenario_case}
  Car traffic scenario: \MODIF{the red car represents the}
  \MODIF{controled vehicle, the blue car is an obstacle (and is fixed in}
  \MODIF{this scenario)}
  }
\end{figure}

In Fig.~\ref{fig:0} the capture basin is represented in blue and for different times.
At time $t=0$ only the points which are in the target area
(and away from the road boundary) are thus represented. 
Then the evolution of the union of (backward) reachable sets is represented for different times.
At time $t_i$, the region 
$Cap(t_i)$ represents the set of starting points 
which can reach the target avoiding the obstacle within $t_i$ seconds.

Next, in Fig.~\ref{fig:1}, 
we have represented the initial position of the vehicle as well 
as the reconstruction of the optimal trajectory (as black line):
the car (red rectangle) drives from left to right in Fig.~\ref{fig:1} 
and it overtakes the fixed obstacle.

The uncolored region (white part) in Fig.~\ref{fig:1} corresponds to the area of starting points from which the blue car cannot reach the target,
whatever maneuver is undertaken (mainly, its velocity is too high to avoid a collision with the blue car). Hence the trajectories
from these points are infeasible, and such starting points do not belong to the backward
reachable sets.


\begin{figure}[!hbtp]
\begin{center}
\includegraphics[width=6cm,angle=0]{\figs/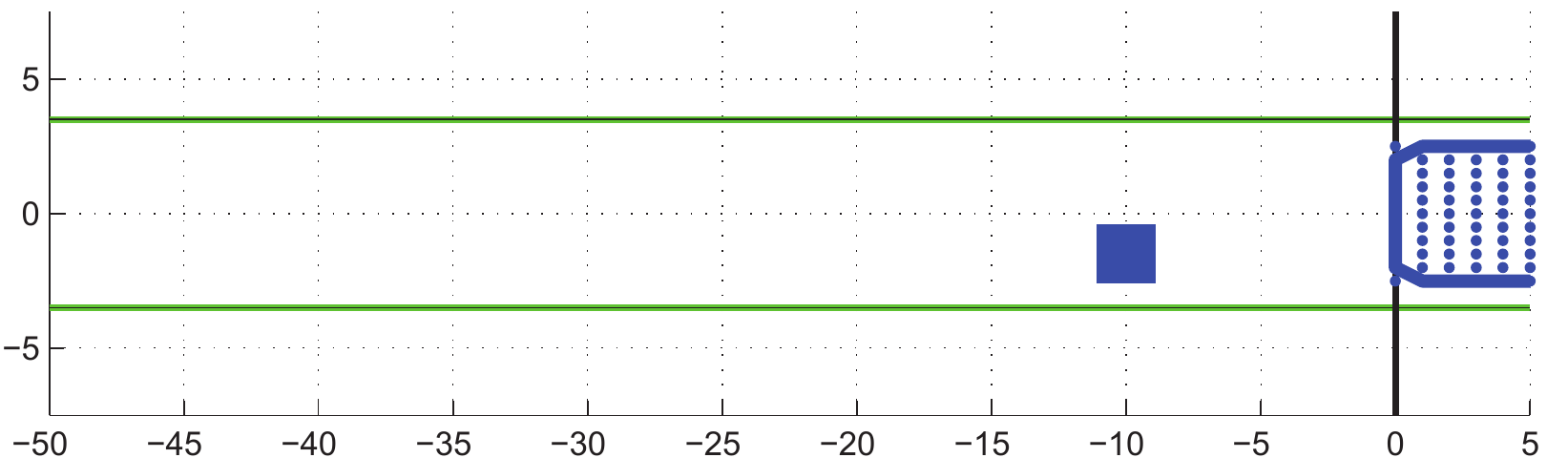} 
\includegraphics[width=6cm,angle=0]{\figs/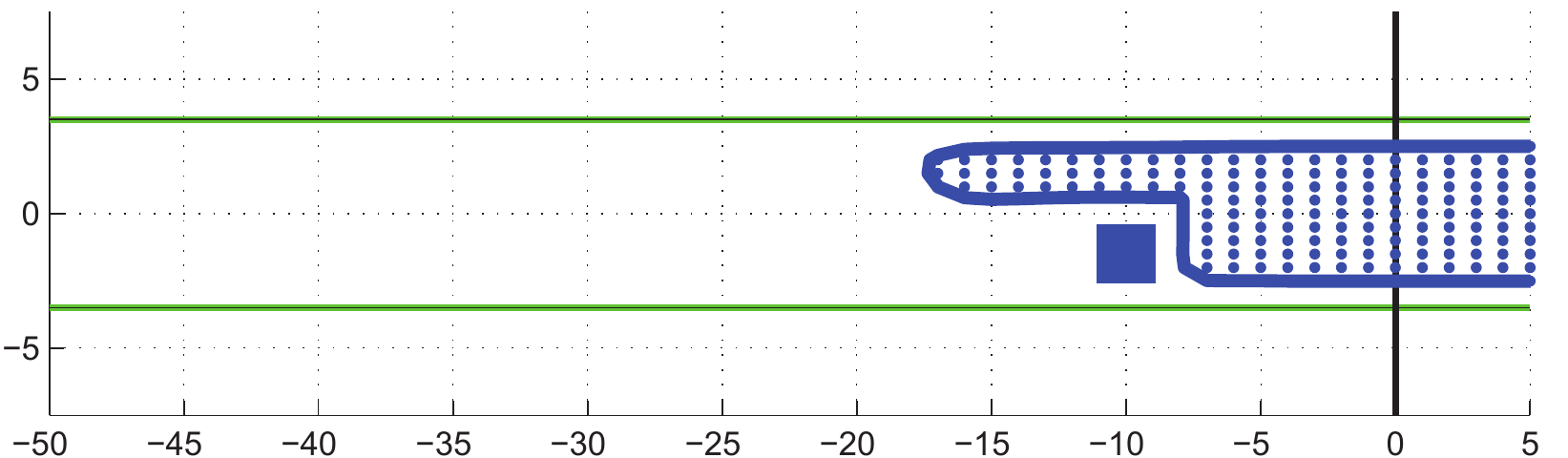} \\
\includegraphics[width=6cm,angle=0]{\figs/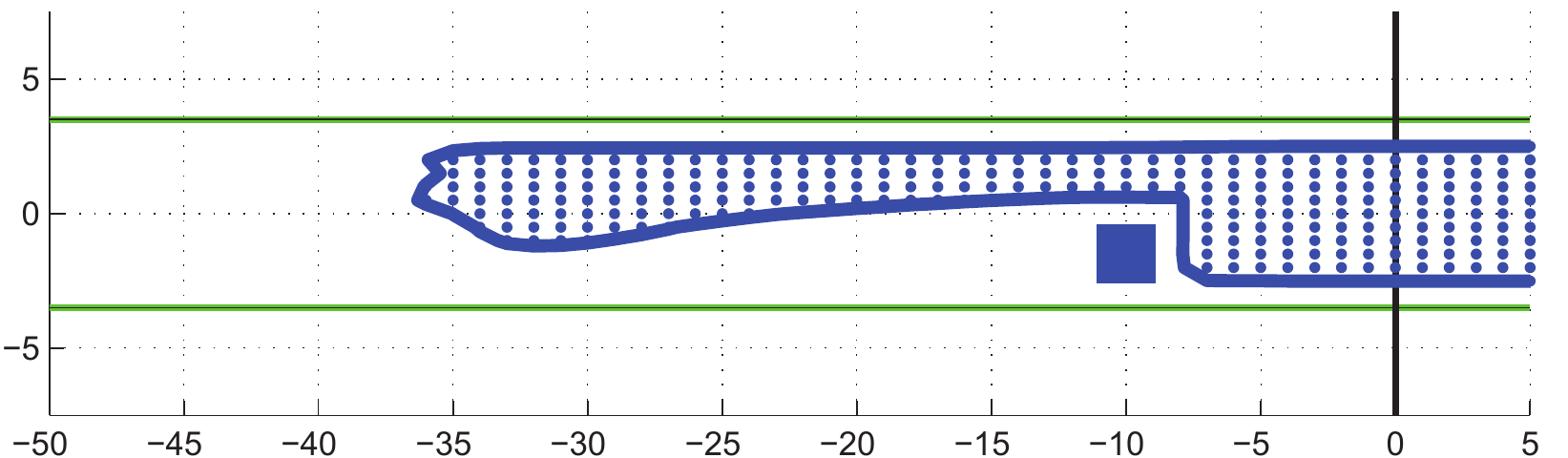} 
\includegraphics[width=6cm,angle=0]{\figs/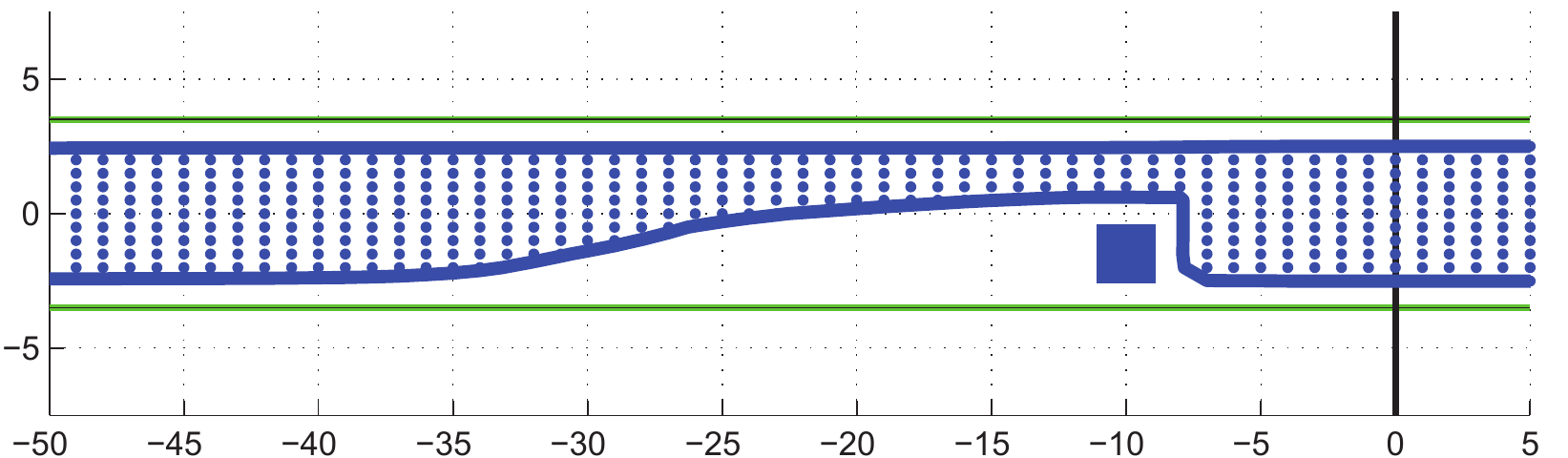} 
\end{center}
\caption{\label{fig:0} (Scenario 1) capture basins $Cap(t_i)$ for different times
($t_1=0s, t_2=0.5s, t_3=1s, t_4=1.5s$) 
}
\end{figure}

\begin{figure}
\includegraphics[width=10cm,angle=0]{\figs/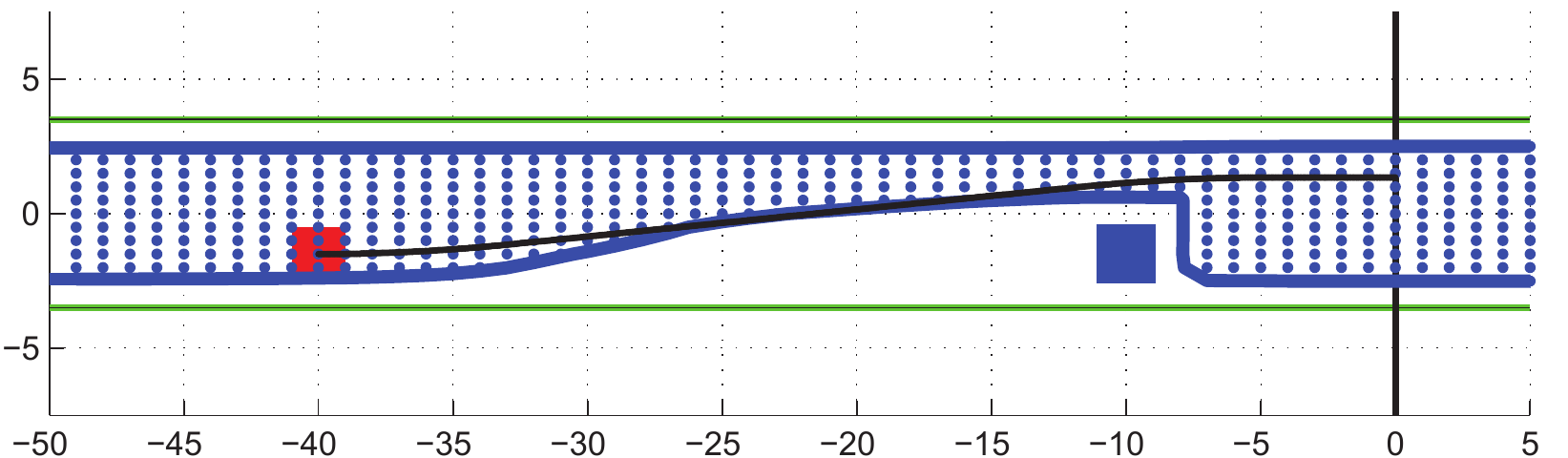} 
\caption{\label{fig:1} (Scenario 1) capture basin $Cap(2)$
and optimal trajectory. 
}
\end{figure}

\paragraph{\bf Convergence test (scenario 1)}
For testing the stability of the HJ approach we first perform a convergence analysis with respect to mesh grid refinement.
We define a grid on the state space \MODIF{(or computational domain)}
\be
  (x,y,\psi,v)\in [-50,10]\times[-4,4]\times [-1,1]\times[5,65]. 
\ee
A variable number of grid points in the $(x,y)$ variables is used, given by 
$$
  \mbox{$N_x=35\cdot 2^m$ and $N_y=4\cdot 2^m$,}
$$
depending on an integer parameter $m\in\{1,\dots,5\}$.
The number of grid points in the $\psi$ and $v$ variables are fixed and given by
$$
  N_{\psi}=20 \quad\mbox{and}\quad N_v=6.
$$
The errors are computed by using a reference value function $v_{ex}$ obtained for $m=5$ 
(i.e., $N_x=1120$, $N_y=128$). 
Furthermore the following CFL (Courant-Friedriech-Levy) restriction 
\be\label{eq:CFL}
  \Delta t\ \sum_{j=1}^4 \frac{\|f_j\|_\infty}{\Delta x_j} \leq \frac{1}{2}
\ee
is used for the stability of the finite difference scheme
(where $f_j$ are the components of the dynamics $f$ and $\|f_j\|_\infty$ is the supremum norm of $f_j$ on the computational domain,
$\Delta z:=(\Delta x_j)_{1\leq j\leq 4}\equiv (\Delta x, \Delta y, \Delta\psi, \Delta v)$ are the different grid mesh steps).

The results are given in Table~\ref{tab:1}. 
For a given grid mesh $(z_i)$
and corresponding vector of mesh sizes $\Delta z$, the local error at grid point $z_i$ is $e_i:= v(z_i)- v_{ex}(z_i)$ 
and the $L^\infty$, $L^1$ and $L^2$ errors are defined as follows:
\be
 e_{L^\infty}:=\max_{i} |e_i|,\quad
 e_{L^1}     :=|\Delta z|\, \sum_{i} |e_i|,\quad
 e_{L^2}     :=|\Delta z|^{1/2}\, \bigg(\sum_{i} e_i^2\bigg)^{1/2}
\ee
where
$|\Delta z|:=\prod_{1\leq j \leq 4} |\Delta x_j|$. 

In Table~\ref{tab:1}, 
in order to evaluate numerically the order of convergence for a given $L^p$ norm,
the estimate $\alpha_m:=\frac{\log(e^{(m-1)}/e^{(m)})}{\log(2)}$ is used for 
corresponding values $N_x=35\cdot 2^m$ and $N_y = 4\cdot 2^m$
(i.e., the mesh steps $N_x$ and $N_y$ are refined by $2$ between two successive computations).

We observe a convergence of order roughly $2$ even for the ENO2-RK1 scheme which in principle 
is only first order in time. This is due to the fact that the dynamics is close to a linear one
in this case (we have also tested a similar RK2 scheme, second order  in time, which gives similar convergence results
on this example).

\begin{table}[!hbtp]
\begin{tabular}{|cc|c|cc|cc|cc||c|}
\hline 
 & & & & & & & &
 & \small CPU \\
 \small $N_x$ & \small $N_y$  & \small $\Dt$
 & \small $e_{L^\infty}$   & \small order 
 & \small $e_{L^1}$        & \small order 
 & \small $e_{L^2}$        & \small order 
 & \small time (in s) \\[0.1cm]
\hline
\hline 
  70 &  8 & 3.97 E-3 & 0.489  &   --  & 1.126  &  --  & 6.769 &  --  &  0.34 \\ \hline
 140 & 16 & 2.03 E-3 & 0.078  & 2.64  & 0.337  & 1.73 & 2.255 & 1.58 &  1.50 \\ \hline
 280 & 32 & 1.02 E-3 & 0.026  & 1.60  & 0.118  & 1.52 & 0.795 & 1.50 &  9.20 \\ \hline
 560 & 64 & 0.51 E-3 & 0.006  & 2.07  & 0.030  & 1.98 & 0.207 & 1.94 & 69.40 \\ \hline
\end{tabular}
\vspace{0.1cm}
\caption{\label{tab:1} (Scenario 1) error table for varying $(N_x,N_y)$ parameters}
\end{table}

\if{
\begin{remark}
\deleted[id=Olivier]{
Taking into account furthermore a safety margin of size $\eps=20\,cm$, 
according to Lemma~\ref{lem:avoidance2} we can guarantee that 
a collision does not occur at any time
for obstacles of dimension $1.6\,m$ \added[id=Robert]{(in each $x$- and $y$-direction)}
in order to be feasible 
(REM OF R.B: Unclear remark: Lemma~\ref{lem:avoidance2} needs a value for $\bar v$,
the obstacle dimensions appear only in Lemma~\ref{lem:condition2} (which is wrong).
Furthermore, is it an additional remark for Scenario 1?)
}
\end{remark}
}\fi

\paragraph{\bf Comparison with a direct method (scenario 1)}
For this particular scenario, we validate the results
by comparing with a direct optimal control approach for calculating the reachable set.
The simulations are obtained by using the OCPID-DAE1 Software~\cite{ocode},
and following the approach described in~\cite{bai-ger-xau-2013} and~\cite{ifip,mtns}.
The resulting capture basin for time $T=2$ is plotted in Fig.~\ref{fig:2} 
(upper graph) and it is in good correspondence with the set obtained by the HJ approach (Fig.~\ref{fig:2}, lower graph).
Notice that we can only expect that both computed capture basins are
equal up to some accuracy of the order $O(\Dx)$ (with $\Dx=1$ in this figure).


\begin{remark}\label{rem:5.2}
{\em 
The advantage of using a direct optimal control approach (like the OCPID-DAE1 Software)
is that it is able to deal with a greater number of 
states variables, which is necessary whenever we need a precise car-model, close to the behavior of a real car.
However the handling of state constraints (in particular obstacles with nonsmooth boundaries)
sometimes leads to numerical difficulties in order to compute feasible trajectories. 

On the other hand the pde solver 
(like ROC-HJ for solving Hamilton-Jacobi equations) is limited, in practice, by the number of state variables,
because it requires to solve a pde with as many dimensions as the number of state variables.
However if the dimension can be processed on a given computer system, then
the HJ approach requests only the Lipschitz property of the functions describing the dynamics and the state constraints.
In particular, there is no problem for dealing with nonsmooth obstacles such as rectangular obstacles, crossing roads scenarios
(more generally, it could handle polygonal roads or more complex polytopial obstacles).
}
\end{remark}

\if{
\begin{figure}[H]
\includegraphics[width=10cm,angle=0]{\figs/brsOC-eps-converted-to.pdf} 
\includegraphics[width=10cm,angle=0]{\figs/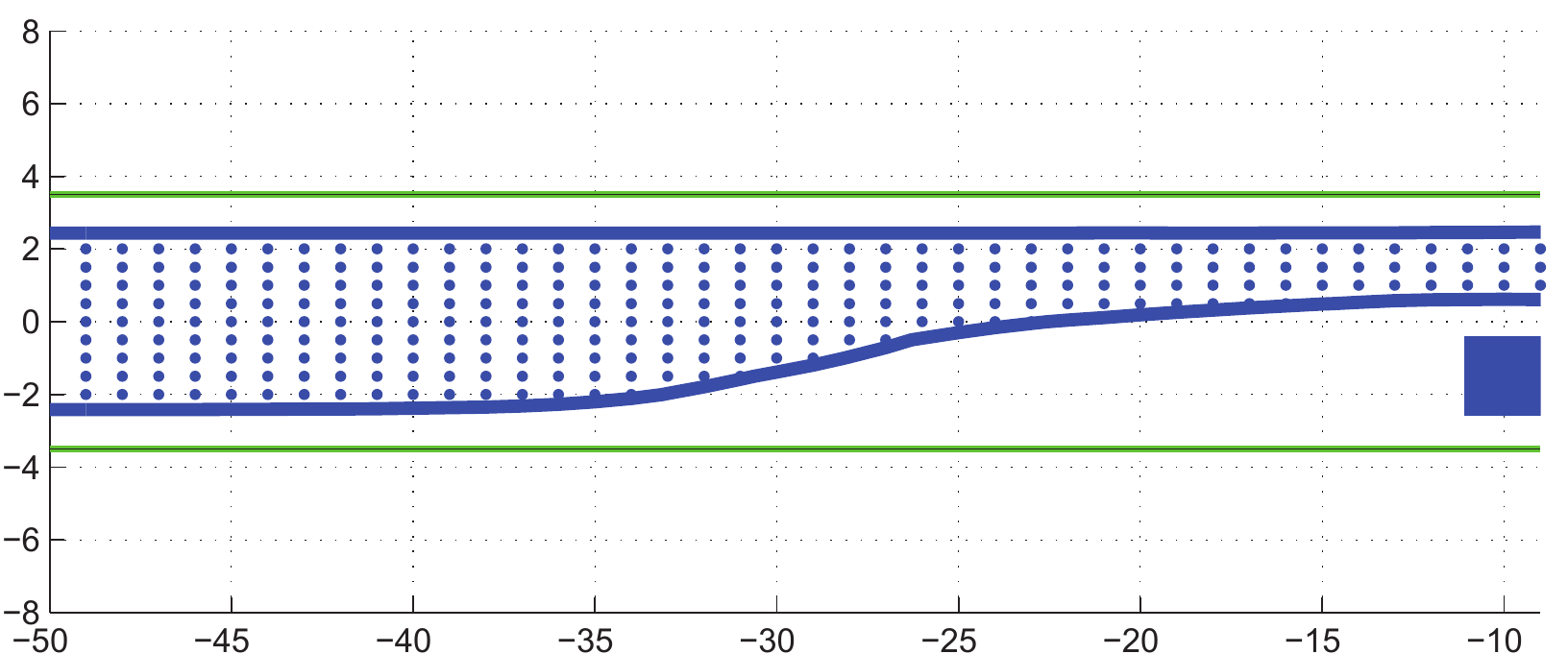}
\caption{\label{fig:2} (Scenario 1) comparison of a direct method (left) with the HJB approach (right)}
\end{figure}
}\fi


\begin{figure}[!h]
  \hspace{-0.1cm}\includegraphics[width=10.1cm,angle=0]{\figs/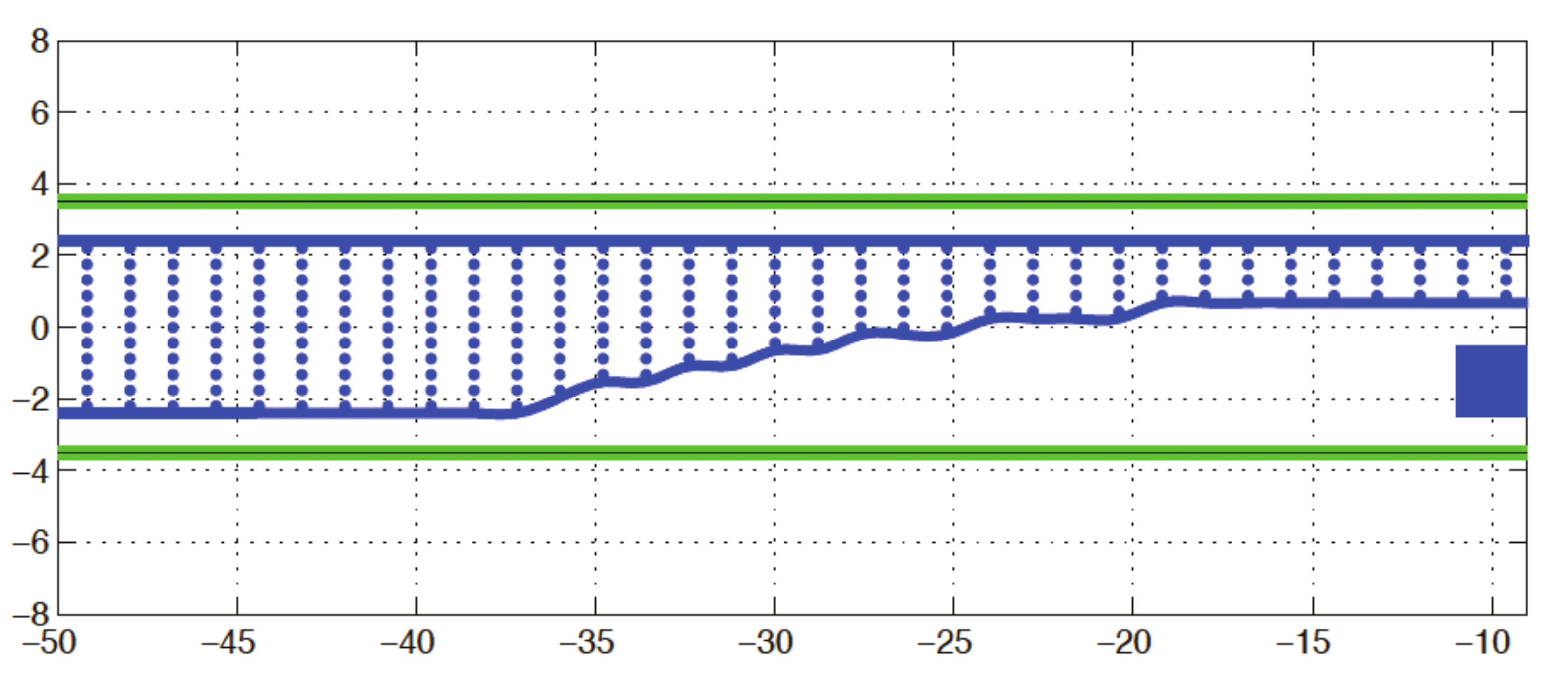}
\includegraphics[width=10cm,angle=0]{\figs/brsHJ-eps-converted-to.pdf}
\caption{\label{fig:2} (Scenario 1) comparison of capture basins obtained with a direct method (up) and the HJ approach (down)}
\end{figure}




Next, we consider more complex scenarios and different road geometries.

\subsection{Scenario 2: straight road with varying width}
We shall consider a highway road with varying width : 
\begin{equation*}\left\{
\begin{array}{l}
y_{up}=3.5~m\\
y_{down}= -3.5~m, \mbox{ if } z_1\leq -15.0\\
y_{down}=-7.0~m, \mbox{ if } z_1 > -15.0. 
\end{array}
\right.
\end{equation*}
This is illustrated in Figs.~\eqref{set_straightlarger} and \eqref{set_straightlarger_more} (the road is represented with green lines).
This can be interpreted as an additional exit lane appearing only for a short part of the considered road. 

Two obstacles (blue rectangles) are moving with linear motion (see \eqref{eq:obs_motion_line}) 
in the same direction as the reference vehicle (red rectangle). All object widths and lengths are here equal to $1~m$.
The set of blue points depicted in Figure~\ref{fig:ex_sl3} and Figure~\ref{fig:ex_sl2} is the projection on the $(x,y)$-panel
of the capture basin for $t_f=2~s$, 
with yaw angle $\psi(t_0)=0$ and velocity $v(t_0)=35~m\,s^{-1}$. 

\medskip

{\em Scenario 2a:} (see Fig.~\ref{fig:ex_sl3})
In this example an overtaking maneuver is considered with one obstacle (first car) in front of the vehicle,
moving forward with velocity $10\,ms^{-1}$ and to be overtaken,
and a second obstacle (second car) next to the vehicle 
also moving forward but with higher velocity $20\,m s^{-1}$ and blocking the maneuver.
In Fig.~\ref{fig:ex_sl3}, the initial position of the vehicle and 
of the obstacle cars are depicted at the initial time $t_0=0$.
The parameters used in the computation for this figure are the grid with
 $(N_x,N_y)=(70,12)$; the trajectory (black line) is starting from
$(x(0),y(0))=(-40.0,\, -1.5)$, $\psi(0)=0$ and $v(0)=35 ms^{-1}$, a
first obstacle car takes initial values $(x(0),y(0))=(-10,\,-1.5)$, with $\psi(0)=0$ and a constant velocity $v(0)=10\,ms^{-1}$;
a second obstacle car takes initial values
$(x(0),y(0))=(-40,\, 1.5)$, with $\psi(0)=0$ and a constant velocity $v(0)=20\,ms^{-1}$.    

\begin{figure}[!hbtp]
\includegraphics[width=10cm,angle=0]{\figs/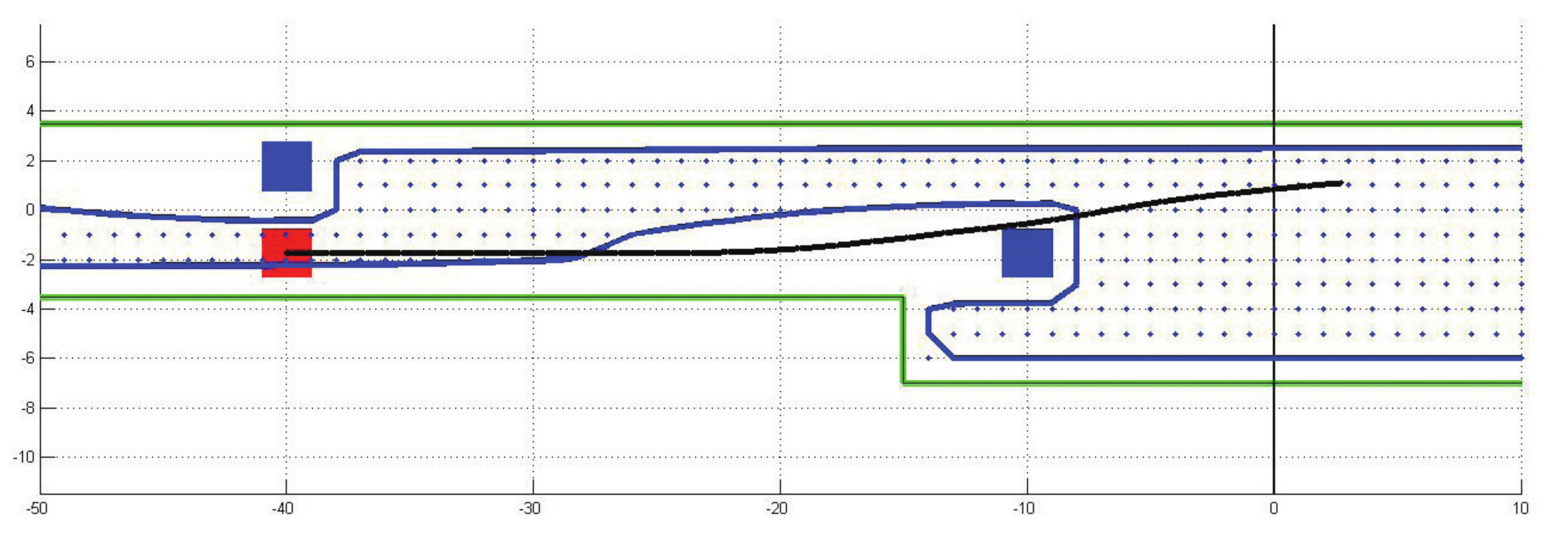}
  \caption{\label{fig:ex_sl3}(Scenario 2a) capture basin $Cap(2)$. 
  \MODIF{Here the blue car above the vehicle (red car) is moving 
  forward (to the right) at constant velocity $20\, ms^{-1}$, the second blue car also moves forward at velocity $10\,ms^{-1}$,
  while the controled}
  \MODIF{
  vehicle (red car) has an initial velocity of $35\,ms^{-1}$.}
}
\end{figure}

\medskip

{\em Scenario 2b:} (see Fig.~\ref{fig:ex_sl2})
This example is similar to the previous one \MODIF{except for} the fact that the second car is
next to the first one at initial time.
The second obstacle car now takes initial values $(x(0),y(0))=(-10,\, 1.5)$, and other parameters are otherwise unchanged.


\begin{figure}[!hbtp]
\includegraphics[width=10cm,angle=0]{\figs/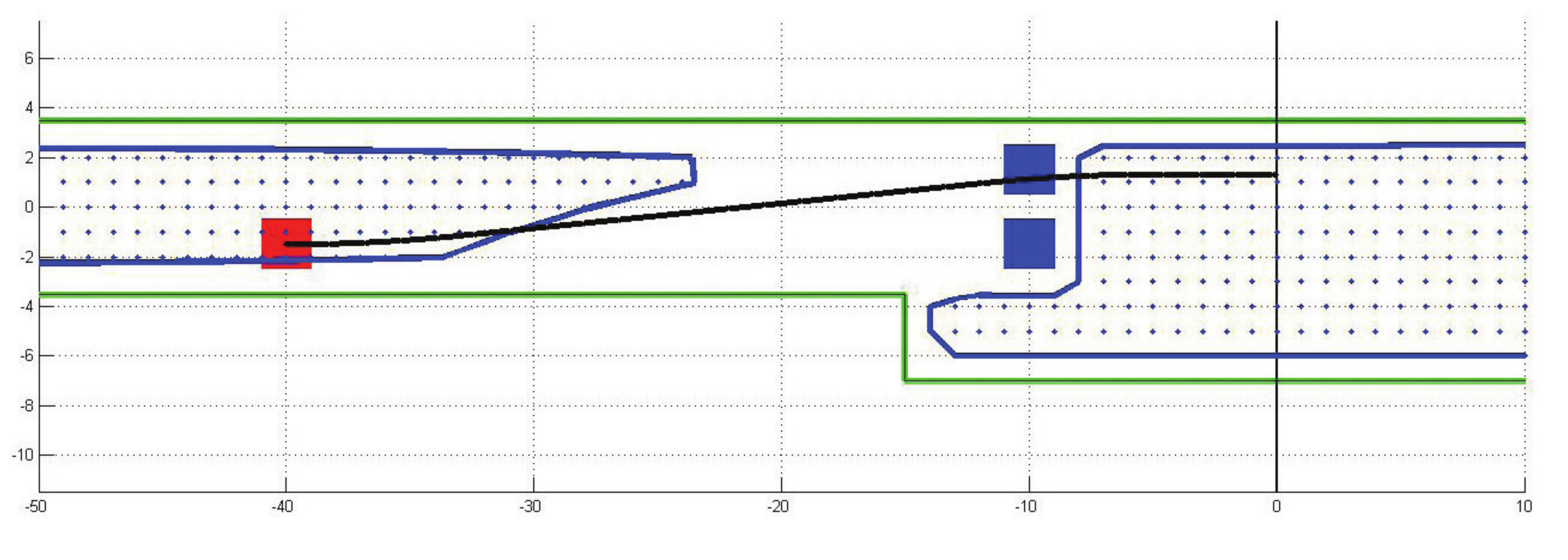}
\caption{\label{fig:ex_sl2}(Scenario 2b) 
{capture basin $Cap(2)$, non-connected, in contrary to Fig.~\ref{fig:ex_sl3}}
}
\end{figure}

The capture basin is the set of initial points of $\R^4$ (according to the model \eqref{eq:model}) 
for which the collision can be avoided and the target can be reached within $t_f$ seconds
(so that the final conditions $x(t_f)\geq 0$ and $\psi(t_f)=0$ are satisfied). 
The set of blue points depicted in Figure~\ref{fig:ex_sl2} (also in Figure~\ref{fig:ex_sl3})
is the projection on the $(x,y)$-plane 
of the capture basin, where the yaw angle and the velocity are fixed to the values
$\psi(t_0)=0$ and $v(t_0)=35~m\,s^{-1}$ (which corresponds to the initial values of the car when the maneuver starts).

By starting in the blue region, the vehicle (in red) can avoid a collision. 
On the other hand, starting from a point in the white area will lead to infeasibility, i.e., the 
vehicle will either go outside the road or will collide with the obstacle
before being able to reach the target area.

In both figures we notice an extra part of the backward reachable set which lays below the second obstacle.
This is due to the varying road width, and it means that the vehicle 
may also start from the exit lane.

In Figure~\eqref{fig:ex_sl2} the capture basin is not connected
which means that the vehicle (red rectangle) can avoid a collision by starting the maneuver either leaving the obstacles 
behind (since it is faster no crash will occur), by bypassing the second car from the exit lane
or by starting from a sufficiently large distance
behind the two obstacles depending on its $y$ position 
(about $24\,m$ if the reference vehicle starts in the first lane and $14\,m$ if it starts in the second lane).

The optimal trajectory (black line) seems to overlap the obstacles and their trajectories 
before reaching the target set. This is because the blue rectangles only show the initial position 
of the obstacles at time $t_0$, and not the evolution of their linear motions in the time interval $[t_0,t_f]$.

\subsection{Scenario 3: curved road with fix or moving obstacles}
The road shape is now described by the set of equations~\eqref{eq:set_curve} with a $7$~m width and a road radius of $50~m$; 
the boundaries of the road as shown by green lines in Figure~\ref{fig:curve2}.  
The initial velocity of the vehicle is set to $v(t_0)=30\,ms^{-1}$. 

{\em Scenario 3a:}
Two fixed obstacles (blue rectangles with different width and length parameters)
have to be avoided by the vehicle (red square), see Fig.~\ref{fig:curve2}, 
and the target $x\geq 0$, $\psi=0$ has to be reached, if possible, in the time interval $[0,5]$ measured in seconds.
The first obstacle has dimensions $0.5~m$ and it is positioned in $(-5,48.25)$. The second obstacle has width $0.5~m$ and length $1~m$, its position is $(-25,45)$.

\begin{figure}[!hbtp]
\begin{center}
\includegraphics[width=7cm,angle=0]{\figs/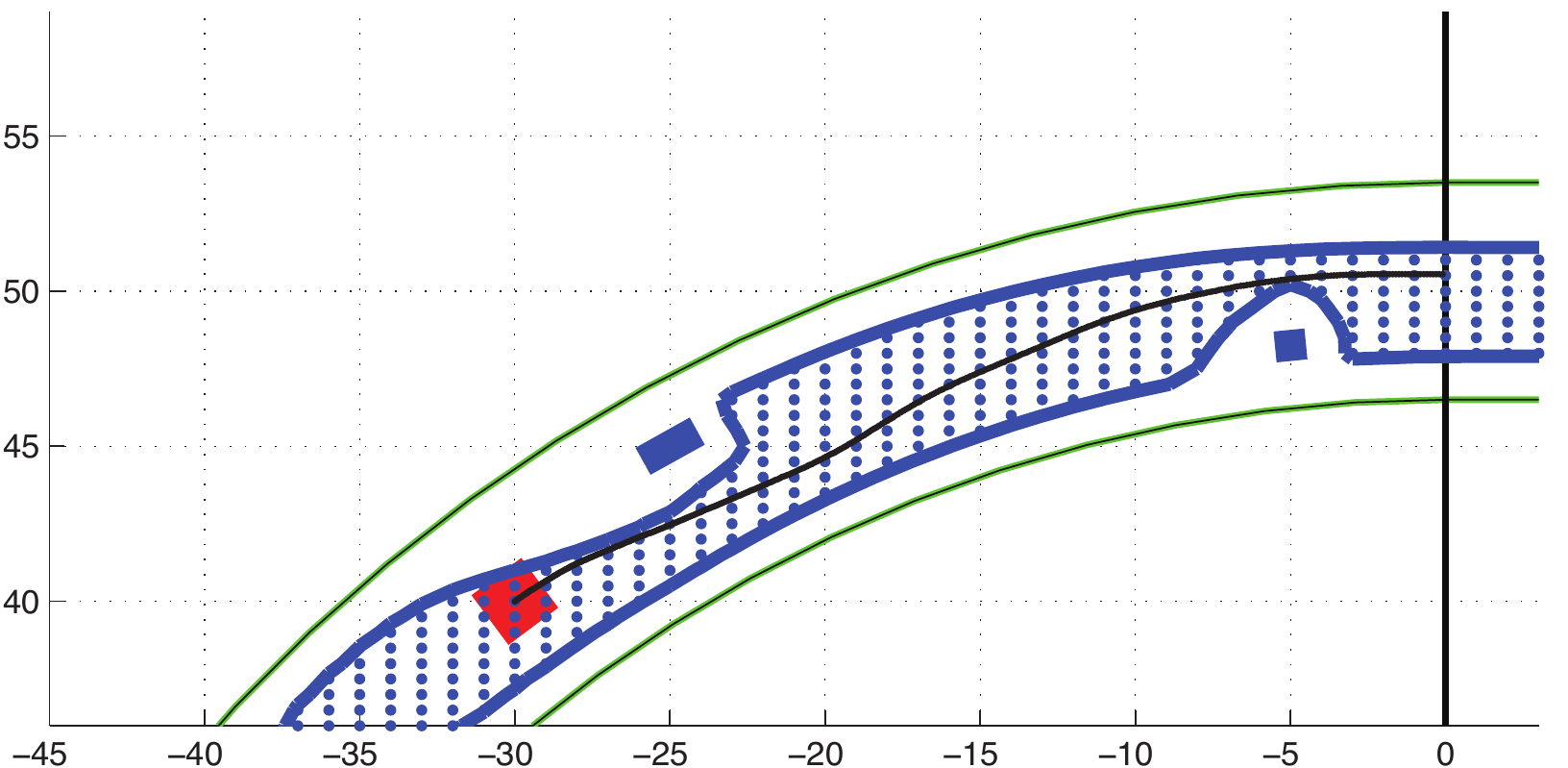}
\end{center}
\caption{\label{fig:curve2}(Scenario 3a) 
capture basin $Cap(5)$
for a curved road and two fixed obstacles
}
\end{figure}

{\em Scenario 3b:}
In this scenario depicted in Figure~\ref{fig:curve1}, 
the road parameters are similar, there is now only one obstacle but it is furthermore moving with a circular motion at speed of $5\,ms^{-1}$.

\begin{figure}[!hbtp]
\begin{center}
\includegraphics[width=7cm,angle=0]{\figs/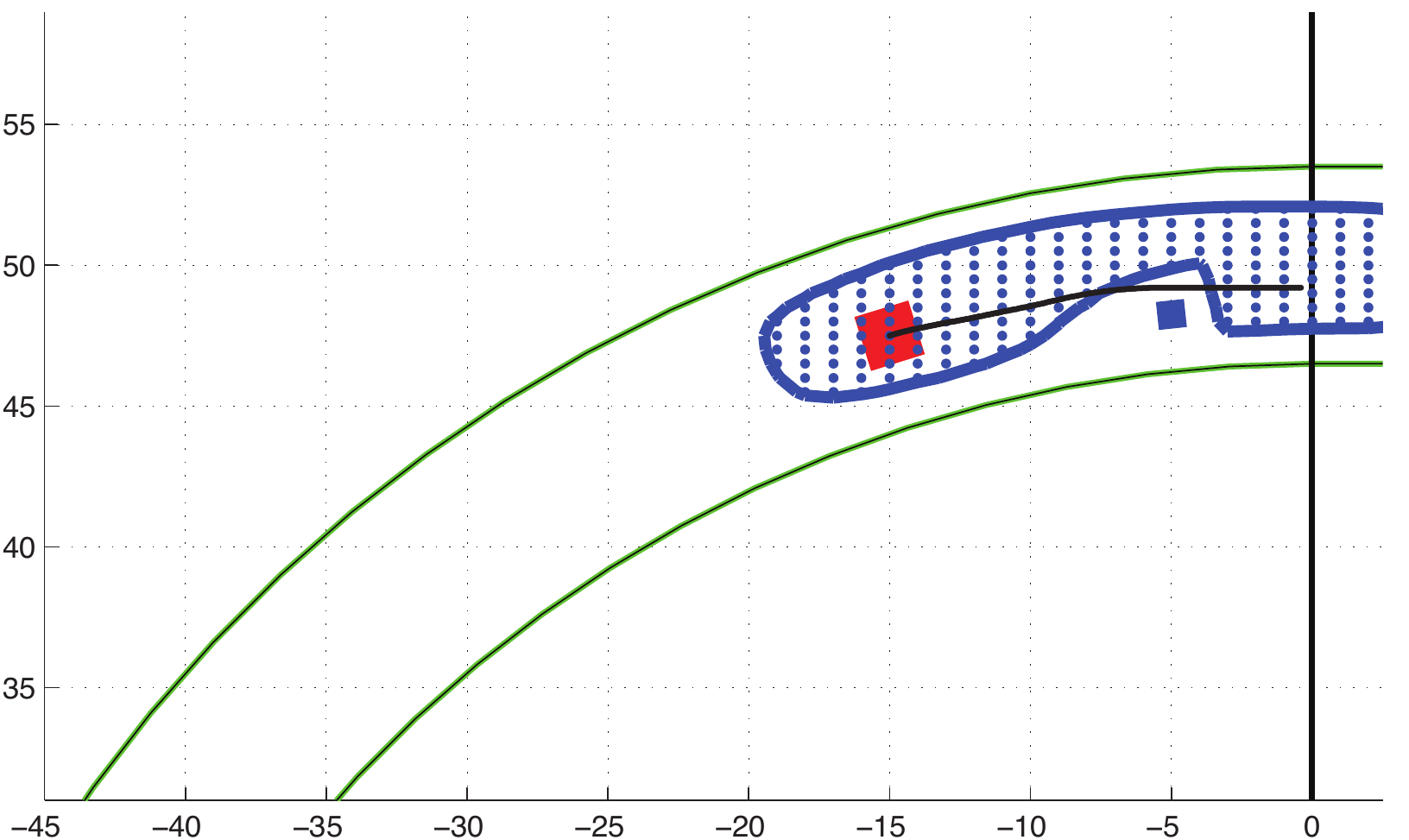}
\end{center}
\caption{\label{fig:curve1}(Scenario 3b) capture basin $Cap(5)$
for a curved road and one moving obstacle
}
\end{figure}

\subsection{Scenario 4: crossing road and moving obstacles} 

The following scenario involves a crossing. As in \eqref{funct_crossing} the width of the four streets 
involved in the crossing can be different.

Here, the horizontal lower and upper road bounds and the vertical limits are different, as
illustrated in Fig.\ref{fig:ex_cr2}.
An object (blue rectangle) of dimensions $1\,m$ is traveling from left to right from position $(-10.0,-2.0)$ meters 
with speed $5\,m s^{-1}$ and deceleration $5\,m s^{-2}$ (until a stop at $t=1s$). 
A second obstacle (length $1.0~m$ and width $2.0~m$) starting from position 
$(-18,4)$
is traveling from top to bottom with speed $5\,ms^{-1}$ and deceleration 
$5\,ms^{-2}$ (until a stop at $t=1s$).
Within time $t_f=2.5\,s$ the red square of dimensions $1.0\,m$ has to reach one of the three targets 
at the end of each road: top (with yaw angle $\frac{\pi}{2}$), bottom 
(with yaw angle $-\frac{\pi}{2}$) or right (with yaw angle $0$). 
At the initial speed  $v(t_0)=20.0\,ms^{-1}$ and for an initial position as in~Fig.\ref{fig:ex_cr2}, 
the red vehicle is able to leave the crossing before the second obstacle enters the center of the crossing. 
In this example the optimal trajectory (black line) will steer to overtake
the obstacle in the front and it will also decelerate 
to avoid the second obstacle.


\begin{figure}[!hbtp]
\begin{center}
\includegraphics[width=5cm,angle=0]{\figs/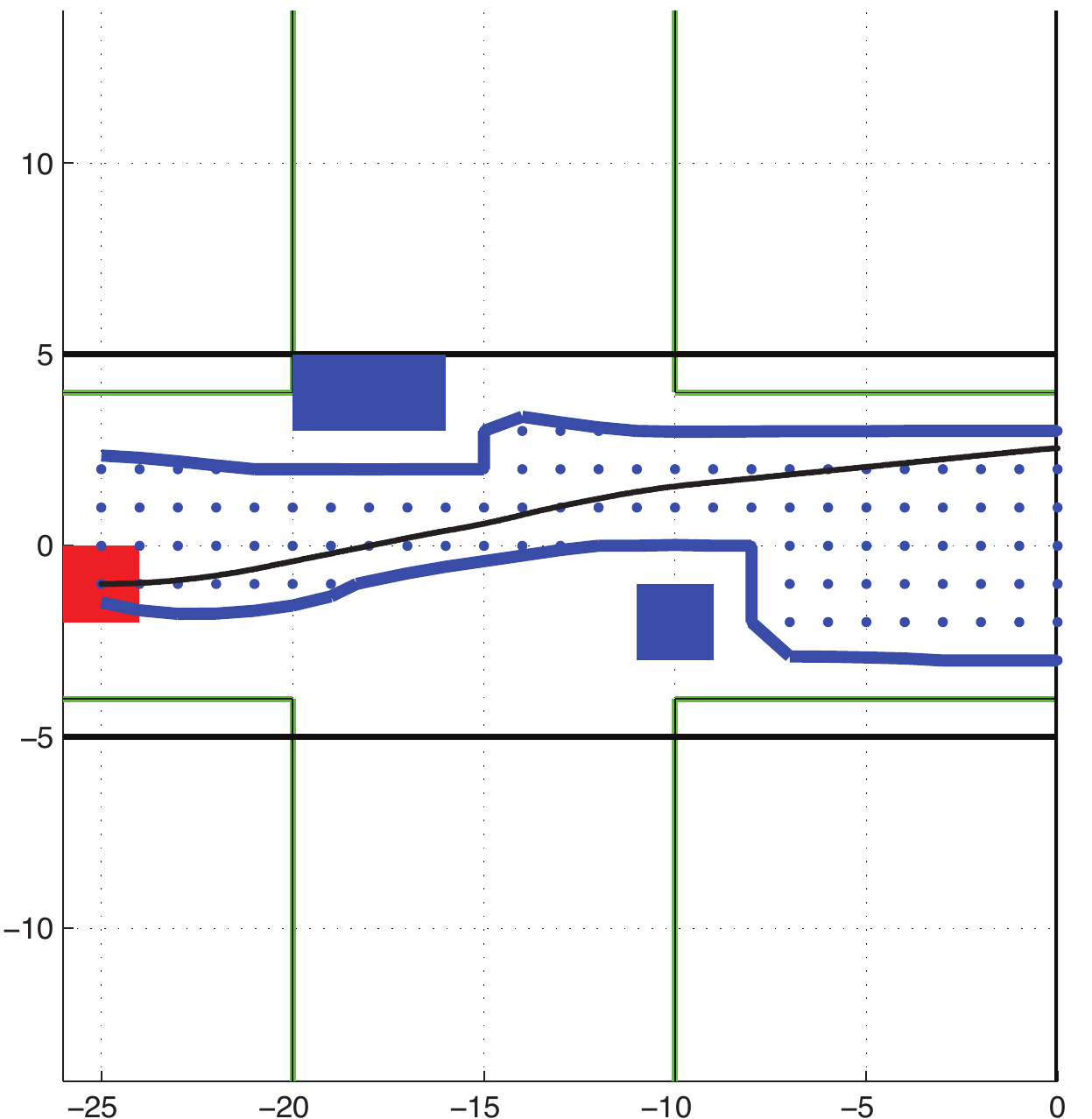}
\end{center}
\caption{\label{fig:ex_cr2} (Scenario 4)
capture basin $Cap(2.5)$ for a crossing with one fixed and one moving obstacle
}
\end{figure}

\section{Conclusion}
In this paper, we have shown the feasibility of the HJB approach
for computing target regions for vehicle collision avoidance problems.

\MODIF{
We have modelized the vehicle avoidance problem by using a 4-dimensional "point mass" model to describe the vehicle, with different
road geometries as well as different obstacles (fixed or moving obstacle vehicles).
We make use of level-set functions in order to represent the reachable sets, the obstacles (road boundaries), the obstacle avoidance
(between possibly moving vehicles). This complex situation leads in general to non-convex reachable sets. These sets
can be represented by using level sets of functions encoded on a grid mesh.
The HJB approach turns out to be a powerful tool especially for complicated road geometries and multiple obstacles,
and can handle general nonlinear dynamics.
An avoidance procedure for the avoidance of rectangular vehicles has been justified in detail within the HJB framework.
}

A next and challenging step would be to 
analyze the present approach using more precise models, 
such as the  7-dimensional "single track" model in~\cite{mtns,IlariaPHD}.
Ongoing works also concern the sensitivity of the secure region with respect to small disturbances of the data.

\ifUseBibTeX
\bibliographystyle{plain}
\bibliography{abbreviations_short,references}
\else
\def\cprime{$'$} \def\cydot{\leavevmode\raise.4ex\hbox{.}}
  \def\cfac#1{\ifmmode\setbox7\hbox{$\accent"5E#1$}\else
  \setbox7\hbox{\accent"5E#1}\penalty 10000\relax\fi\raise 1\ht7
  \hbox{\lower1.15ex\hbox to 1\wd7{\hss\accent"13\hss}}\penalty 10000
  \hskip-1\wd7\penalty 10000\box7} \def\Dbar{\leavevmode\lower.6ex\hbox to
  0pt{\hskip-.23ex \accent"16\hss}D} \def\dbar{\leavevmode\hbox to
  0pt{\hskip.2ex \accent"16\hss}d}

\fi

\if{
\appendix\label{app:scheme}
}\fi

\end{document}